\documentclass[reqno]{amsart}

\usepackage{verbatim} 

\usepackage[utf8]{inputenc}
\usepackage[T1]{fontenc}
\usepackage[english]{babel}
\usepackage{lmodern}

\usepackage{ragged2e}

\usepackage{eurosym}

\usepackage{enumitem}
\usepackage{graphicx}
\usepackage{float}
\usepackage{amsmath}
\allowdisplaybreaks[4]
\usepackage{mathtools}
\usepackage{amssymb}
\usepackage{amsthm}
\usepackage{bbm}

\newcommand{\N}{\mathbb{N}}

\newcommand{\R}{\mathbb{R}}

\newcommand{\pr}{\mathbb{P}}	
\newcommand{\ept}{\mathbb{E}}	

\newtheoremstyle{normal}
  {}
  {}
  {}
  {}
  {\bfseries}
  {}
  {.5em}
  {}

\theoremstyle{plain}
\newtheorem{theorem}{Theorem}

\newtheorem{corollary}[theorem]{Corollary}
\newtheorem{lemma}[theorem]{Lemma}

\newtheorem{assumption}{Assumption}

\theoremstyle{definition}

\newtheorem{remark}[theorem]{Remark}
\allowdisplaybreaks[4]						

\usepackage{color}

\usepackage[a4paper,bindingoffset=0.2in,%
            left=0.5in,right=0.5in,top=1in,bottom=1in,%
            footskip=.25in ]{geometry}

\usepackage{setspace}

\usepackage{fancyhdr}

\usepackage{ifthen}

\allowdisplaybreaks[4]

\fancyheadoffset{\textwidth}
\usepackage[pdftex,pdfpagelabels=true]{hyperref}

\hypersetup{
    colorlinks=true,
    linkcolor=blue,
    filecolor=magenta,
    urlcolor=cyan,
}

\begin{document}

\title[]{Mean-field derivation of a two-dimensional signal-dependent parabolic-elliptic Keller-Segel system in algebraic scaling}

\date{\today}
\thanks{This work is supported from the Deutsche Forschungsgemeinschaft (DFG, German Research Foundation) -- 547277619 and National Nature Science Foundation of China - $11571181$}

\author{Lukas Bol$^\ast$} 
\thanks{$^\ast$Corresponding author.}
\address[Lukas Bol]{School of Business Informatics and Mathematics, Universität Mannheim, 68131, Mannheim, Germany}
\email{lukas.bol@uni-mannheim.de}

\author{Li Chen}
\address[Li Chen]{School of Business Informatics and Mathematics, Universität Mannheim, 68131, Mannheim, Germany}
\email{lichen@uni-mannheim.de}

\begin{abstract} This paper continues our survey about the mean-field derivation of the two-dimensional signal-dependent Keller-Segel system studied in \cite{BOL2026113712}. Therefore, we consider the same system of moderately interacting particles as before. The difference lies in the scaling. Since logarithmic scaling was treated in \cite{BOL2026113712}, we now consider algebraic scaling to obtain propagation of chaos in the weak sense. We prove convergence in probability for the particle trajectories. Moreover, for short times and regularity assumptions on the initial data we show the convergence of the densities in the $L^1$ norm. The novelty of this paper is the treatment of a particle model (with algebraic scaling) where the moderate interaction completely takes place in the diffusive term. This structure with algebraic scaling makes tremendous difference from the propagation of chaos discussion when the interaction appears in the drift terms with singular kernel in \cite{chenwangwang2025}. The argument for convergence in probability proceeds by defining a stopping time based on a power mean of the sample paths. Furthermore, we prove that the convergence in probability with this power-mean is equivalent to the convergence with maximum norm on trajectories.
\end{abstract}

\maketitle

\pagenumbering{arabic}

\section{Introduction}
In this paper we rigorously derive the the two dimensional signal-dependent Keller-Segel system \eqref{singalKellerSegel}, studied in \cite{BOL2026113712}, from a moderately interacting particle system.
\begin{align}
	\label{singalKellerSegel}
	\begin{cases}
		\partial_t u=\Delta( e^{-v}u+ u),    \\
		-\Delta v+v=\chi u, \\
		u(0,x)=u_0(x), \qquad x\in \R^2,\;\; t>0.
	\end{cases}
\end{align}
The function $u$ describes a density of bacteria or cells and $v$ represents a concentration of chemical signals (molecules). 
The Signal-dependent Keller-Segel system describes the phenomena that bacteria or cells orient their movement according to the concentration of a chemical indication, as well as the intensity, gradient strength, or temporal variation of the signal itself. As an improvement of classical Keller Segel system with a fixed chemotactic sensitivity, signal-dependent effect in the modeling includes also nonlinear functions of the signal into the chemotactic flux. 
The expression "signal-dependent" in \eqref{singalKellerSegel} refers to the factor $e^{-v}$ which contains the influence of the chemical signals on the motion of the bacteria. The positive constant $\chi$ determines the strength rate of signals. These models are widely used to study bacterial chemotaxis, cell migration in tissues, and self-organization in microbial populations. The well-posedness of signal-dependent Keller-Segel system (with more general forms) has been intensively studied in the last decade on bounded domains, \cite{tao2014energy,tao2017effects,burger2020delayed,jin2020critical,fujie2020global,fujie2021comparison,fujie2022global,jiang2022boundedness} to name a few.
For more discussions of signal dependent models and the existence of solutions we refer to \cite{BOL2026113712}, where a global well-posedness of \eqref{singalKellerSegel} and its mean-field derivation have been studied on $\R^2$ with logarithmic scaling. 

We stress on a few relevant aspects here. The first equation of \eqref{singalKellerSegel} features mass conservation. So when the initial condition $u_0$ is a probability density, it is ensured that the solution remains a probability density for all times. This is crucial in order to compare $u$ to a law of stochastic particles. Using the Yukawa potential $\tilde \Phi$ (see \cite[Ch.~6.23]{lieb_loss}), the second equation can be restated as $v= \Phi \ast u$, where $\Phi := \chi \tilde \Phi$. The Yukawa potential enjoys the following properties
\begin{align} \label{prop_Phi}
\nabla  \tilde \Phi \in L^p(\R^2), p \in [1,2) \text{ and } \tilde \Phi \in L^q(\R^2), q \in [1,\infty) .
\end{align}
However, $\Phi$ is singular at the origin and so we consider a mollification $\Phi^\varepsilon:= \Phi * j^\varepsilon$ for our particle system, where $j^\varepsilon (x) := \frac{1}{\varepsilon^2 } j(x / \varepsilon)$ is a standard mollifier. Concretely, our particle model looks like this:
\begin{align}
\label{generalized_regularized_particle_model}
\begin{cases}
dX^\varepsilon_{N,i}
= \Big(2\exp\Big(-\displaystyle\frac{1}{N}\displaystyle\sum_{j=1}^N\Phi^\varepsilon(X_{N,i}^\varepsilon-X^\varepsilon_{N,j}) \Big)+2\Big)^{1/2} dB_i(t) ,  \\
X^\varepsilon_{N,i}(0) =\zeta_i,\qquad 1\leq i\leq N.
\end{cases}
\end{align}
The main purpose of this paper is to derive \eqref{singalKellerSegel} from \eqref{generalized_intermediate_particle_model} as a limit for $N \to \infty$ and $\varepsilon \to 0$. In \cite{BOL2026113712} we did that for logarithmic scaling ($\varepsilon = c(\log N)^{-\beta}$). Actually, the cut-off range $\varepsilon$ with logarithmic scaling is so rough that the singularity from $\Phi$ does not make the proof complicated. Now we treat the case of algebraic scaling ($\varepsilon = N^{-\beta}$ for some $\beta >0$). We work in the standard setting for stochastic differential equations. We take some complete filtered probability space $(\Omega, \mathcal{F}, (\mathcal{F}_{t\geq 0}) , \mathbb{P})$ which supports $2$-dimensional $\mathcal{F}_t$-Brownian motions $\{ (B_i(t))_{t\geq 0} \}_{i=1}^\infty$ and random variables $\{\zeta_i\}_{i=1}^\infty$, which are all independent. Moreover, we assume that the distribution of every $\zeta_i$ has density $u_0$ (so $\{\zeta_i\}_{i=1}^\infty$ are i.i.d.). Finally our particles are described by the solutions  $X^\varepsilon_{N,i} : [0,T] \times \Omega \to \R^2$ of \eqref{generalized_regularized_particle_model}. 

For a general overview of mean-field theory we refer to \cite{BOL2026113712,chen2024fluctuationsmeanfieldlimitattractive} and references therein.When the interaction appears in the diffusion term, there are also classical results for regular interactions, such as in \cite{FERNANDEZ199733}. However, for singular interactions, such as in \eqref{generalized_regularized_particle_model}, there are so far only results with the logarithmic scaling. We refer to \cite{chen2022analysis} for cross-diffusion system and to \cite{BOL2026113712} for the signal dependent Keller-Segel equation. For the singular interaction in drift terms, there are a sequence of breakthroughs in the last decade, for example in \cite{Lazarovici2017mean, serfaty2018systems, bresch2023mean, oliverarichardtomasevic2023, chen2025quantitative}. We summarize some recent results on mean-field limit with singular interaction for algebraic scaling. Lazarovici and Pickl derive the Vlasov–Poisson system from a deterministic particle model with random initial data in \cite{Lazarovici2017mean}, where they use a cut-off to regularize coulomb type potentials and show convergence in the sense of probability. The following works consider stochastic particle systems. In \cite{huangluipickl2020} Huang, Liu and Pickl obtain convergence in the sense of probability for a derivation of the Vlasov–Poisson–Fokker–Planck system. Fluctuation for a model with (attractive) Riesz kernels is studied by Chen, Holzinger and Jüngel in \cite{chen2024fluctuationsmeanfieldlimitattractive}, including a convergence in probability result. In \cite{chenwangwang2025} Chen, Y. Wang and Z. Wang consider a mean-field control problem for a Keller-Segel system. Therefore they need to show convergence in probability and propagation of chaos in the strong sense, by using the relative entropy method.

In all the works mentioned for singular interaction in drift terms, the diffusive term is just additive noise. To our knowledge there are no results about models where the singular interaction (with algebraic scaling) takes place in the diffusive term. The main challenge comes from the estimate on the diffusion term, which is from structure point of view very different from the singular interaction appears in the drift term case.
 We want to compare the trajectories of the interacting particles and the particles of the limiting system. In case of additive noise the diffusive terms cancel and only a deterministic integral remains, whereas in the situation of the diffusive case, the stochastic integral still exists. In the smooth interaction case (or singular interaction with logarithmic scaling), for each particle dynamics one can easily apply the Burkholder-Davis-Gundy inequality for martingales. While for singular interaction with algebraic scaling, the maximum (for example in \cite{Lazarovici2017mean, chen2025mean,chenwangwang2025}) among all the particles makes the dynamics not any more a (local-) martingale.

In order to formulate and prove our statements, we need to introduce some additional equations. The first step can be seen as formally taking the limit $N \to \infty$ of \eqref{generalized_regularized_particle_model}, while $\varepsilon$ is fixed, which leads to the intermediate particle model
\begin{align}
\label{generalized_intermediate_particle_model}
\begin{cases}
d\bar{X}^\varepsilon_{i} = \big(2\exp(-\Phi^\varepsilon\ast u^\varepsilon(t, \bar{X}^\varepsilon_{i}))+2\big)^{1/2} dB_i(t) ,\\
\bar{X}^\varepsilon_{i}(0) =\zeta_i,\qquad 1\leq i\leq N .
\end{cases}
\end{align}
Here $u^\varepsilon(t)$ is the probability density of the distribution of $\bar{X}^\varepsilon_{i}(t)$, thus \eqref{generalized_intermediate_particle_model} is of Mckean-Vlasov type.

The advantage is the decoupling of the system \eqref{generalized_regularized_particle_model}. The particles in \eqref{generalized_intermediate_particle_model} do not interact and are consequently independent due to our assumptions on the initial values and Brownian motions. It can be easily obtained by It$\hat{o}$'s formula that the probabilistic density $u^\varepsilon$ solves the following non-local partial differential equations
\begin{align}
\label{kellersegelmedium}
\begin{cases}\partial_t u^\varepsilon =  \Delta ( e^{-{v^\varepsilon}}u^\varepsilon+ u^\varepsilon),     \\
-\Delta {v^\varepsilon}+{v^\varepsilon}=\chi u^\varepsilon\ast j^\varepsilon,  \\
u^\varepsilon(0,x)=u_0(x), \qquad x\in \R^2,\;\; t>0.
\end{cases}
\end{align}
At this point we want to mention since $u^\varepsilon(t )\geq 0$ we have $e^{-v^\varepsilon}=e^{-\Phi^\varepsilon \ast u^\varepsilon}$ is bounded, for each fixed $\varepsilon$.  
By applying the classical parabolic theory we conclude that system \eqref{kellersegelmedium} has a unique global smooth non-negative solution $(u^\varepsilon,{v^\varepsilon})$.
In a second step we formally let $\varepsilon \to 0$ in \eqref{generalized_intermediate_particle_model} and arrive at the limiting particle model:
\begin{align}
\label{generalized_particle_model}
\begin{cases} d\hat{X}_{i} =   \big(2\exp(-\Phi\ast u(t, \hat{X}_{i}))+2\big)^{1/2} dB_i(t),  \\
\hat{X}_{i}(0) =\zeta_i,\qquad 1\leq i\leq N,
\end{cases}
\end{align}
which also is of Mckean-Vlasov type since $u(t)$ is the probabilistic density  of the distribution of $\hat{X}_{i}(t)$.
To obtain Lipschitz continuity for $\Phi*u$, $u$ has to compensate the singularity of $\Phi$. Sufficient regularity of $u$ is guaranteed in \cite[Theorem 2]{BOL2026113712} under the appropriate assumptions on the initial data $u_0$. 

\begin{assumption} \label{ass_ks} \text{}
\begin{itemize}
\item[a)] $u_0$ is a probability density which satisfies $ u_0\log u_0\in L^1(\R^2)$ and $ u_0\in L^1(\R^2, |x|^2\,dx)\cap L^p(\R^2)$ for all $p \in [1,\infty)$.
\item[b)] The Brownian motions $\{ (B_i(t))_{t\geq 0} \}_{i=1}^\infty$ and random variables $\{\zeta_i\}_{i=1}^\infty$ are all independent. The density of the distribution of $\zeta_i$ is $u_0$.
\item[c)] $\chi<4/c_*$ where $c_*$ is the optimal constant in Sobolev inequality: $\|w\|_{L^4(\R^2)}^4\leq c_*\|w\|_{L^2(\R^2)}^2\|\nabla w\|_{L^2(\R^2)}^2$.
\end{itemize}
\end{assumption}

For the existence of the solutions of \eqref{generalized_intermediate_particle_model} and \eqref{generalized_particle_model} we refer to the discussion after theorem 3 in \cite{BOL2026113712}.

Our first theorem compares the particles $X^\varepsilon_{N,i}$ of system \eqref{generalized_regularized_particle_model} to the particles $\bar{X}^\varepsilon_{i}$ of the intermediate model \eqref{generalized_intermediate_particle_model}, which is the most elaborate part of this work. We explicitly estimate the probability that the trajectories differ from each other, allowing algebraic scaling for $\varepsilon$ and $N$. As theorem \ref{ks_conv_of_max_X-bar_X} states, this probability becomes small with arbitrary fast rate.

\begin{theorem} \label{ks_conv_of_max_X-bar_X}
Let $T >0$. Under the assumptions \ref{ass_ks} it holds: \\
$\forall \alpha \in (0, \frac{1}{2} ) : \forall \beta \in (0, \frac{\alpha}{2} ) :\forall \gamma >0 : \exists C>0 : \forall \varepsilon >0, N \in \N, \text{ with } \varepsilon \geq N^{-\beta} : $
\begin{align*}
\sup_{t \in [0,T]} \pr \left( \left\lbrace \max_{1 \leq i \leq N}  \left| X^\varepsilon_{N,i} - \bar{X}^\varepsilon_{i} \right|(t) \geq N^{-\alpha} \right\rbrace \right) \leq  \frac{C}{N^\gamma}.
\end{align*}
\end{theorem}

In principle, the structure of the proof follows the ideas used for singular interaction in drift terms, such as in \cite[Thm 1.4]{lichenzhang2024}, \cite[Thm 1]{chen2024fluctuationsmeanfieldlimitattractive} or \cite[Prop 4.1]{chenwangwang2025}. In those works the authors consider a stopped process to compare the trajectories. When we look at the expression inside the probability, we can obviously raise the power on both sides without changing the probability. By choosing sufficiently high powers one can obtain arbitrary fast convergence rates. To prevent the error of the trajectories from exploding, they introduce a stopping time. The final conclusion is done by a Gronwall argument. The necessary estimations involve a law of large number result in order to compare the particle interaction and the mean-field interaction.
This careful treatment enables algebraic scaling. 

However, the models in those works mentioned above only feature additive noise as diffusive term. In our case (see \eqref{generalized_regularized_particle_model}) we have to deal with a singular interaction in the diffusive noise. So there is no cancellation and the error $X^\varepsilon_{N,i} - \bar{X}^\varepsilon_{i}$ is a stochastic integral. This requires the use of the Burkholder-Davis-Gundy (BDG) inequality. But this inequality can only be applied to (local-)martingales. Although $X^\varepsilon_{N,i} - \bar{X}^\varepsilon_{i}$ are martingales for $1 \leq i \leq N$, the maximum in $i$ is not any more. For this reason, we can not follow the idea used in the literature mentioned directly. To give a clear explanation of our approach we introduce what is sometimes called power mean or p-mean:
\begin{align*}
M_p(x_1,\dots,x_N):= \bigg( \frac{1}{N} \sum_{i=1}^N \left| x_i \right|^{p} \bigg)^{\frac{1}{p}}, \enspace p \geq 1, \enspace  M_\infty(x_1,\dots,x_N):= \max_{1 \leq i \leq N}  \left| x_i \right|.
\end{align*}
The suggestive notation is no coincidence, since $M_p$ approximates $M_\infty$ as $p \to \infty$. Instead of $M_\infty(X^\varepsilon_{N,i} - \bar{X}^\varepsilon_{i})$ we work with $M_p(X^\varepsilon_{N,i} - \bar{X}^\varepsilon_{i})$ for large $p$. Inside the probability the root can be shifted. Once we moved on to expectation using Markov's inequality, we can bring the sum outside. The only thing left inside the expectation is some power of a martingale, which can be handled by the  Burkholder-Davis-Gundy inequality. In previous works the maximum was raised by some additional power. In our case this power is already included in $M_p(X^\varepsilon_{N,i} - \bar{X}^\varepsilon_{i})$. For this reason, the stopping time also has to depend on $p$. So the parameter $p$ (later in the proof denoted by $k$) becomes part of the statement, since we show the convergence in probability of $M_p(X^\varepsilon_{N,i} - \bar{X}^\varepsilon_{i})$ for $p$ large enough. In the final step, we can switch to $M_\infty(X^\varepsilon_{N,i} - \bar{X}^\varepsilon_{i})$ without any loss in terms of the range of the parameters, since we work with probabilities. \\

To complete the propagation of chaos result, we need to estimate $\bar{X}^\varepsilon_{i}-\hat{X}_{i}$ in order to prove theorem \ref{ks_conv_of_max_X-hat_X}. At this point, we need a slight extension of \cite[Proposition 15]{BOL2026113712}, which is provided by lemma \ref{conv_ept}.

\begin{theorem}[Convergence in probability] \label{ks_conv_of_max_X-hat_X}
Let $T >0$. Under the assumptions \ref{ass_ks} it holds: \\
$\forall \eta \in (0,1) : \forall \beta \in (0, \frac{1}{4} ) :\forall \gamma >0 : \exists C>0 : \forall N \in \N, \text{ with } N \geq 2 : $
\begin{align*}
\sup_{t \in [0,T]} \pr \bigg( \Big\{ \max_{1 \leq i \leq N}  \left| X^\varepsilon_{N,i} - \hat{X}_{i} \right|(t) \geq \varepsilon^\eta\Big\} \bigg) \leq  C \varepsilon^\gamma ,
\end{align*}
where $\varepsilon = N^{-\beta}$.
\end{theorem}

Theorem \ref{ks_conv_of_max_X-hat_X} implies the weak propagation of chaos.

\begin{corollary}[Propagation of chaos in the weak sense]
 Let $k \in \N$, $T >0$ and $\beta \in (0, \frac{1}{4} )$. We set $\varepsilon_N:=N^{-\beta }$. Under the assumptions \ref{ass_ks} it holds for the joint distribution for every $t \in [0,T]$
$$ \pr^{(X^{\varepsilon_N}_{N,1}(t),\dots ,X^{\varepsilon_N}_{N,k}(t))} \rightharpoonup u(t)^{\otimes k} \lambda^{2k}, \text{ as } N \to \infty ,$$
where $\lambda^{2k}$ denotes the Lebesgue measure on $\R^{2k}$.

\end{corollary}

In \cite{BOL2026113712} we already proved propagation of chaos in the strong sense for logarithmic scaling using relative entropy and convergence in expectation. Now we follow the idea of \cite{chenwangwang2025}
to use convergence in probability (theorem \ref{ks_conv_of_max_X-bar_X}) in order to obtain the strong $L^1$ propagation of chaos in algebraic scaling.

\begin{theorem}[Propagation of chaos in the strong sense] \label{poc_strong}
\label{relative} \text{} \\
Let $T >0$ and the assumptions \ref{ass_ks} hold. If we further assume that there exists some $C>0$ such that \\
 $\|\nabla\log u^\varepsilon\|_{L^\infty(0,T;L^\infty(\mathbb{R}^2))}
 \leq C$ for all $\varepsilon >0$ it holds \\  
 $\forall l \in \N : \forall \alpha \in (0,\frac{1}{2}) : \forall \beta \in (0,\frac{\alpha}{2}): \forall b \in (0,2\frac{\alpha}{\beta}-4 ) : \exists C>0 : \forall N \geq l, \varepsilon \in (0,1] \text{ with } \varepsilon \geq N^{-\beta} : $
\begin{align*}
\|u^\varepsilon_{N,l}-u^{\otimes l}\|^2_{L^\infty(0,T;L^1(\mathbb{R}^{2l}))}\leq  C \max\big\{ \varepsilon^{b} , \varepsilon\big\}.
\end{align*}
Here $u^\varepsilon_{N,l}(t)$ denotes the $l$-th marginal density of the joint density $u^\varepsilon_N(t)$ of the distribution of $(X^\varepsilon_{N,i}(t))_{1\leq i\leq N}$. \\
\end{theorem}

\begin{remark}
The assumption $\|\nabla\log u^\varepsilon\|_{L^\infty(0,T;L^\infty(\mathbb{R}^2))} \leq C$ can be guaranteed by  \cite[Thm 5]{BOL2026113712} for short times $T$ provided that $u_0 \in L^\infty (\R^2)$ and $\nabla \log u_0 \in W^{1,p}(\mathbb{R}^2))$ for some $p>2$.
\end{remark}

\begin{remark}
Due to $\varepsilon \leq 1$ only the interval $[\frac{2}{5}\alpha,\frac{\alpha}{2})$ is relevant for $\beta$. Since $\beta < \frac{2}{5}\alpha$ implies $2\frac{\alpha}{\beta}-4 >1$, we can choose $b >1$. However, this has no effect on the asymptotics, because $\max\big\{ \varepsilon^{b} , \varepsilon\big\} = \varepsilon $. 
\end{remark}

The arrangement of the paper is the following: The proof of theorem \ref{ks_conv_of_max_X-bar_X} is given in section 2. We first show the convergence in probability for $k$-mean of the trajectory difference for big enough $k$, which is the first main technical novelty of this work. Then we prove in section 3 that this convergence is equivalent to the convergence in probability for the maximum trajectory difference. In section 4, we finish the propagation of chaos result both in weak and strong senses. We use the convergence of trajectory difference to estimate the rest terms in the relative entropy estimates.

\section{Convergence in the power mean}

As has been mentioned in the introduction, we first prove that the $k-$mean of the trajectory difference converges in the sense of probability only for big enough $k$. 
\begin{theorem} \label{conv_of_mean} 
Let $T >0$. Under the assumptions \ref{ass_ks} it holds: \\
 $ \forall \alpha \in (0,\frac{1}{2} ) :  \forall \beta \in (0,\frac{\alpha}{2}) : \forall \gamma >0 : \forall k \in ( \frac{1}{\alpha-2\beta} \vee \frac{2\gamma}{1-2\alpha}, \infty ) : \exists  C>0 : \forall \varepsilon >0, N \in \N  \text{ with } \varepsilon \geq N^{-\beta} :$ 
$$ \sup_{t \in [0,T]} \pr \bigg( \Big\{ \Big( \frac{1}{N} \sum_{i=1}^N \left| X^\varepsilon_{N,i} - \bar{X}^\varepsilon_{i} \right|^{2k}(t)  \Big)^{\frac{1}{2k}} \geq N^{-\alpha} \Big\} \bigg) \leq  \frac{C}{N^\gamma}. $$
\end{theorem}

\begin{remark} In view of lemma \ref{equi_mean_max}, as will be given later, the explicit dependence of $k$ w.r.t. $\alpha, \beta$ and $\gamma$ is pointless. This dependence appears in the Gronwall argument of theorem \ref{conv_of_mean} and is relevant for the stronger estimation of $\ept (S_{N,k}(t ))$ (defined in the proof).
\end{remark}

\begin{proof} Let $\alpha >0$ and $k \geq 1$. Further let $\varepsilon >0$, $N \in \N $ and $t \in [0,T]$. In the whole proof $C$ denotes a generic constant which may vary in each step and which is independent of $\varepsilon $, $N$ or $t$. (So $C$ is allowed to depend on $k$.)\\

We define the hitting time $\tau_{N,k}$, which depends on $\varepsilon$ as well.
$$ 
\tau_{N,k} := \inf \Big\{  s \in [0,T] \quad \Big| \quad\frac{1}{N} \sum_{i=1}^N \left| X^\varepsilon_{N,i} - \bar{X}^\varepsilon_{i} \right|^{2k}(s) \geq N^{-\alpha 2k} \Big\}
$$
Since $X^\varepsilon_{N,i}$ and $\bar{X}^\varepsilon_{i}$ have continuous sample paths and $[N^{-\alpha},\infty )$ is closed, $\tau_{N,k}$ is a stopping time (see \cite[Ch.~1.2]{karatzas2014brownian}).
Using $\tau_{N,k}$ we define the process $S_{N,k}$
$$ S_{N,k}(s ) := N^{2\alpha k} \frac{1}{N} \sum_{i=1}^N \left| X^\varepsilon_{N,i} - \bar{X}^\varepsilon_{i} \right|^{2k}(s \wedge \tau_{N,k}  ).$$
Again by the continuity of the sample paths one can argue that 
\begin{align} \label{S_N_leq_1}
S_{N,k} \leq 1.
\end{align}
Moreover one can obtain that $ \displaystyle\frac{1}{N} \sum_{i=1}^N \left| X^\varepsilon_{N,i} - \bar{X}^\varepsilon_{i} \right|^{2k}(s ) \geq N^{-\alpha 2k}$ implies $ S_{N,k}(s ) =1$ and we conclude
\begin{align*}
&\pr \bigg( \Big\{ \Big( \frac{1}{N} \sum_{i=1}^N \big| X^\varepsilon_{N,i} - \bar{X}^\varepsilon_{i} \big|^{2k}(t)  \Big)^{\frac{1}{2k}} \geq N^{-\alpha} \Big\} \bigg) = \pr \bigg( \Big\{ \frac{1}{N} \sum_{i=1}^N \big| X^\varepsilon_{N,i} - \bar{X}^\varepsilon_{i} \big|^{2k}(t) \geq N^{-\alpha 2k} \Big\} \bigg) \\
\leq & \pr \big( \left\lbrace S_{N,k}(t ) =1 \right\rbrace \big) \leq \ept (S_{N,k}(t )) =  N^{2\alpha k} \frac{1}{N} \sum_{i=1}^N \ept \left( \left| X^\varepsilon_{N,i} - \bar{X}^\varepsilon_{i} \right|^{2k}(t \wedge \tau_{N,k}  ) \right) 
\end{align*}
The next part of the proof consist of a Gronwall argument for $\ept (S_{N,k})$. As abbreviation we introduce
\begin{align*}
\sigma_{N,i}(s):= \bigg(2\exp\Big(-\displaystyle\frac{1}{N}\displaystyle\sum_{j=1}^N\Phi^\varepsilon((X_{N,i}^\varepsilon-X^\varepsilon_{N,j})(s)) \Big)+2\bigg)^{1/2}  - \Big(2\exp(-\Phi^\varepsilon\ast u^\varepsilon(s,\bar{X}^\varepsilon_{i}(s)))+2\Big)^{1/2}.
\end{align*}
By direct estimate the difference of the dynamics, we obtain
\begin{align*}
& N^{2\alpha k} \frac{1}{N} \sum_{i=1}^N \ept \left( \left| X^\varepsilon_{N,i} - \bar{X}^\varepsilon_{i} \right|^{2k}(t \wedge \tau_{N,k}  ) \right) = N^{2\alpha k} \frac{1}{N} \sum_{i=1}^N \ept \bigg( \Big| \int_0^{t \wedge \tau_{N,k}} \sigma_{N,i}(s) dB_i(s) \Big|^{2k} \bigg) \\
\leq &  N^{2\alpha k} \frac{1}{N} \sum_{i=1}^N \ept \bigg( \sup_{s \leq t \wedge \tau_{N,k}} \bigg| \int_0^{s} \sigma_{N,i}(\tilde s) dB_i(\tilde s) \bigg|^{2k} \bigg) \leq C N^{2\alpha k} \frac{2^{k}}{N} \sum_{i=1}^N \ept \bigg(  \bigg| \int_0^{t \wedge \tau_{N,k}} \sigma_{N,i}^{2}(s) ds \bigg|^k \bigg)
\end{align*}
where the BDG inequality \cite[Ch.~3.3.D]{karatzas2014brownian} has been applied.
Therefore, after applying the Hölder's inequality for the time integral, the estimate is reduced into
\begin{align}\label{sigma2k}
\ept (S_{N,k}(t ))\leq C N^{2\alpha k} \frac{1}{N} \sum_{i=1}^N \ept \Big( \int_0^{t \wedge \tau_{N,k}} \sigma_{N,i}^{2k}(s)  ds  \Big).
\end{align}
From this point, the following discussion is the most tricky part in the proof. In contrast to situations as in \cite{Lazarovici2017mean,chen2024fluctuationsmeanfieldlimitattractive,chenwangwang2025}, where the maximum among the particles w.r.t the index $j=1,\cdots,N$ is considered and thus getting rid of the summation in $j$ is easy, we have to keep everywhere the summation in $j$ and have it decoupled from the interaction force $\displaystyle\sum_{j=1}^N\Phi^\varepsilon((X_{N,i}^\varepsilon-X^\varepsilon_{N,j})$. 

Due to the Lipschitz continuity of functions $\sqrt{2\cdot +2}$ and $\exp(-(\cdot ))$, we estimate $\sigma_{N,i}$ in the following:
\begin{align*}
|\sigma_{N,i} (s  )| \leq & C \bigg| \displaystyle\frac{1}{N}\displaystyle\sum_{j=1}^N\Phi^\varepsilon((X_{N,i}^\varepsilon-X^\varepsilon_{N,j})(s ))   - \Phi^\varepsilon\ast u^\varepsilon(s, \bar{X}^\varepsilon_{i}(s)) \bigg|
\end{align*}
(Observe that the above constant is also independent of $i$.) This implies that
\begin{align*}
 \ept (S_{N,k}(t ))\leq &  C N^{2\alpha k} \frac{1}{N} \sum_{i=1}^N \ept \bigg( \int_0^{t \wedge \tau_{N,k}} \bigg| \displaystyle\frac{1}{N}\displaystyle\sum_{j=1}^N\Phi^\varepsilon((X_{N,i}^\varepsilon-X^\varepsilon_{N,j})(s  ))   - \Phi^\varepsilon\ast u^\varepsilon(s, \bar{X}^\varepsilon_{i}(s  )) \bigg|^{2k}  ds  \bigg) \\
\leq &  C N^{2\alpha k} \frac{1}{N} \sum_{i=1}^N \ept \bigg( \int_0^{t \wedge \tau_{N,k}} \bigg| \frac{1}{N}\displaystyle\sum_{j=1}^N\Phi^\varepsilon((X_{N,i}^\varepsilon-X^\varepsilon_{N,j})(s  )) -\frac{1}{N}\displaystyle\sum_{j=1}^N\Phi^\varepsilon((\bar{X}_{i}^\varepsilon - \bar{X}^\varepsilon_{j})(s)) \bigg|^{2k}  ds  \bigg) \\
& +  C N^{2\alpha k} \frac{1}{N} \sum_{i=1}^N \ept \bigg( \int_0^{t \wedge \tau_{N,k}} \bigg| \frac{1}{N}\displaystyle\sum_{j=1}^N\Phi^\varepsilon((\bar{X}_{i}^\varepsilon - \bar{X}^\varepsilon_{j})(s))  - \Phi^\varepsilon\ast u^\varepsilon(s, \bar{X}^\varepsilon_{i}(s)) \bigg|^{2k}  ds  \bigg) \\
=: & I + II.
\end{align*}
For the term $I$, we compare the $\varepsilon$ dependent pair force $\Phi_\varepsilon$ of the particles of \eqref{generalized_regularized_particle_model} and \eqref{generalized_intermediate_particle_model} by using the Taylor's formula.
\begin{align*}
I=&  C N^{2\alpha k} \frac{1}{N} \sum_{i=1}^N \ept \bigg( \int_0^{t \wedge \tau_{N,k}} \bigg| \frac{1}{N}\displaystyle\sum_{j=1}^N \Big( \Phi^\varepsilon((X_{N,i}^\varepsilon-X^\varepsilon_{N,j})(s  )) - \Phi^\varepsilon((\bar{X}_{i}^\varepsilon - \bar{X}^\varepsilon_{j})(s)) \Bigg) \bigg|^{2k} ds \bigg) \\
\leq &  C N^{2\alpha k} \frac{1}{N} \sum_{i=1}^N \ept \bigg( \int_0^{t \wedge \tau_{N,k}} \bigg| \frac{1}{N}\displaystyle\sum_{j=1}^N \Big( \nabla \Phi^\varepsilon\big((\bar{X}_{i}^\varepsilon - \bar{X}^\varepsilon_{j})(s  )\big)\big[(X_{N,i}^\varepsilon-X^\varepsilon_{N,j} - (\bar{X}_{i}^\varepsilon - \bar{X}^\varepsilon_{j}))(s  )\big] \Big) \bigg|^{2k} ds \bigg) \\
& +  C N^{2\alpha k} \frac{1}{N} \sum_{i=1}^N \ept \bigg( \int_0^{t \wedge \tau_{N,k}}\bigg| \frac{1}{N}\displaystyle\sum_{j=1}^N \Big( \Vert D^2 \Phi^\varepsilon \Vert_{L^\infty (\R^2)} \big| (X_{N,i}^\varepsilon-X^\varepsilon_{N,j} - (\bar{X}_{i}^\varepsilon - \bar{X}^\varepsilon_{j}))(s  ) \big|^2 \Big) \bigg|^{2k} ds \bigg) \\
=: & I_1 + I_2
\end{align*}
The term $I_1$ can be further decomposed into two terms,
\begin{align*}
I_1
\leq &  C N^{2\alpha k} \frac{1}{N} \sum_{i=1}^N \ept \bigg( \int_0^{t \wedge \tau_{N,k}} \Big( \frac{1}{N}\displaystyle\sum_{j=1}^N \left| \nabla \Phi^\varepsilon((\bar{X}_{i}^\varepsilon - \bar{X}^\varepsilon_{j})(s  )) \right| \cdot\left| (X_{N,i}^\varepsilon - \bar{X}_{i}^\varepsilon)(s  ) \right| \Big)^{2k} ds \bigg) \\
& +  C N^{2\alpha k} \frac{1}{N} \sum_{i=1}^N \ept \bigg( \int_0^{t \wedge \tau_{N,k}}\Big( \frac{1}{N}\displaystyle\sum_{j=1}^N \left| \nabla \Phi^\varepsilon((\bar{X}_{i}^\varepsilon - \bar{X}^\varepsilon_{j})(s  )) \right| \cdot\left| (X^\varepsilon_{N,j} - \bar{X}^\varepsilon_{j})(s  ) \right| \Big)^{2k} ds \bigg)\\
 =: & I_{1,1} + I_{1,2}.
\end{align*}
The estimate for $ I_{1,1} $ is relatively simple since the summation in $j$ is not involved in the particle difference. We proceed with the same technique as in \cite[Sect.~4]{chen2024fluctuationsmeanfieldlimitattractive},
 where a law of large numbers argument is used
\begin{align*}
I_{1,1} \leq &  C N^{2\alpha k} \frac{1}{N} \sum_{i=1}^N \ept \left( \int_0^{t \wedge \tau_{N,k}}\left( \left| \nabla \Phi^\varepsilon \right| \ast u^\varepsilon(s) (\bar{X}_i^\varepsilon (s  ))\right)^{2k} \left| (X_{N,i}^\varepsilon - \bar{X}_{i}^\varepsilon)(s  ) \right|^{2k} ds \right) \\
& +  C N^{2\alpha k} \frac{1}{N} \sum_{i=1}^N \ept \left( \int_0^{t \wedge \tau_{N,k}} \big| \text{DMC}_i(s  ) \big|^{2k} \left| (X_{N,i}^\varepsilon - \bar{X}_{i}^\varepsilon)(s  ) \right|^{2k} ds \right)  \\
=: & I_{1,1,1} + I_{1,1,2},
\end{align*}
where, for convenience, we use the shortcut $\text{DMC}_i(s  ) := \displaystyle\frac{1}{N}\sum_{j=1}^N \left| \nabla \Phi^\varepsilon((\bar{X}_{i}^\varepsilon - \bar{X}^\varepsilon_{j})(s  )) \right| - \left| \nabla \Phi^\varepsilon \right| \ast u^\varepsilon(s) (\bar{X}_i^\varepsilon (s  ))$ to write the difference of mean and convolution.

We know from \cite{BOL2026113712} that $u^\varepsilon$ features mass conservation. Together with Young's inequality we can see that \\
$\left| \nabla \Phi^\varepsilon \right| \ast u^\varepsilon \in L^\infty((0,T) \times \R^2)$ and due to the mollification $\left| \nabla \Phi^\varepsilon \right| \ast u^\varepsilon(s)$ is also continuous on $\R^2$. For this reason $\big\Vert \left| \nabla \Phi^\varepsilon \right| \ast u^\varepsilon (s) \big\Vert_{L^\infty(\R^2)}$ agrees with the supremum and together with Fubini's theorem we can easily obtain the estimate for $I_{1,1,1}$:
\begin{align*}
I_{1,1,1} \leq & C N^{2\alpha k} \frac{1}{N} \sum_{i=1}^N \ept \left( \int_0^t \big\Vert \left| \nabla \Phi^\varepsilon \right| \ast u^\varepsilon \big\Vert_{L^\infty((0,T) \times \R^2)}^{2k} \left| (X_{N,i}^\varepsilon - \bar{X}_{i}^\varepsilon)(s \wedge \tau_{N,k}  ) \right|^{2k} ds \right) \\
= & C \big\Vert \left| \nabla \Phi^\varepsilon \right| \ast u^\varepsilon \big\Vert_{L^\infty((0,T) \times \R^2)}^{2k} \int_0^{t} \ept \left(  S_{N,k}(s  ) \right) ds.
\end{align*}
For the estimate of $I_{1,1,2}$, we need to use the law of large numbers argument lemma \ref{lln}. First of all, arguing in the same way as we just did, we receive
\begin{align*}
	I_{1,1,2} 
	\leq &   C \int_0^{t} N^{2\alpha k} \frac{1}{N} \sum_{i=1}^N \ept \left(  \big| \text{DMC}_i(s) \big|^{2k} \left| (X_{N,i}^\varepsilon - \bar{X}_{i}^\varepsilon)(s \wedge \tau_{N,k}  ) \right|^{2k} \right) ds =: C \int^t_0I_{1,1,2}' ds.
\end{align*}
Next, we denote $A^N_0(s):=A^N_0\big(|\nabla \Phi^\varepsilon|, u^\varepsilon(s)\big)$ (see lemma \ref{lln}) and split the sample space into two disjoint subsets $\Omega=A^N_{0}(s)\cup {A^N_{0}(s)}^c$ and likewise $I_{1,1,2}'$.
\begin{align*}
I_{1,1,2}' \leq & N^{2\alpha k} \frac{1}{N} \sum_{i=1}^N \ept \left( \mathbbm{1}_{A^N_{0}(s)} \big| \text{DMC}_i(s) \big|^{2k} \left| (X_{N,i}^\varepsilon - \bar{X}_{i}^\varepsilon)(s \wedge \tau_{N,k}  ) \right|^{2k} \right)\\
& + N^{2\alpha k} \frac{1}{N} \sum_{i=1}^N \ept \bigg( \mathbbm{1}_{{A^N_{0}(s)}^c} \Big( \max_{1 \leq i \leq N} \big|\text{DMC}_i(s  )\big| \Big)^{2k} \left| (X_{N,i}^\varepsilon - \bar{X}_{i}^\varepsilon)(s \wedge \tau_{N,k}  ) \right|^{2k} \bigg) 
\end{align*}
Notice that $\big(\bar{X}_{i}^\varepsilon(s  )\big)_{i=1}^N$ are i.i.d. random variables with law $u^\varepsilon(s)$. Moreover, the mass conservation of $u^\varepsilon$ mentioned earlier in the proof leads to $\Vert \text{DMC}_i \Vert_{L^\infty((0,T) \times \R^2)} \leq 2 \Vert \nabla \Phi^\varepsilon \Vert_{L^\infty(\R^2)}$. Let $k_1 \in \N$. Now we can apply Lemma \ref{lln} and use \eqref{S_N_leq_1} to conclude
\begin{align*}
I_{1,1,2}' \leq &  N^{2\alpha k} \frac{1}{N} \sum_{i=1}^N \ept \left( \mathbbm{1}_{A^N_{0}(s)} \Vert \nabla \Phi^\varepsilon \Vert_{L^\infty(\R^2)}^{2k} \left| (X_{N,i}^\varepsilon - \bar{X}_{i}^\varepsilon)(s \wedge \tau_{N,k}) \right|^{2k} \right)\\
& + N^{2\alpha k} \frac{1}{N} \sum_{i=1}^N \ept \Big( \mathbbm{1}_{{A^N_{0}(s)}^c} \left| (X_{N,i}^\varepsilon - \bar{X}_{i}^\varepsilon)(s \wedge \tau_{N,k}) \right|^{2k} \Big) \\
\leq & \Vert \nabla \Phi^\varepsilon \Vert_{L^\infty(\R^2)}^{2k} \ept \left( \mathbbm{1}_{A^N_{0}(s)} S_{N,k}(s) \right)+ \ept \left( S_{N,k}(s) \right) \\
\leq & \Vert \nabla \Phi^\varepsilon \Vert_{L^\infty(\R^2)}^{2k}  \pr \left( A^N_{0}(s) \right) + \ept \left( S_{N,k}(s  ) \right)\\
\leq & \Vert \nabla \Phi^\varepsilon \Vert_{L^\infty(\R^2)}^{2k} C N^{-k_1+1} \Vert \nabla \Phi^\varepsilon \Vert_{L^\infty(\R^2)}^{2{k_1}} + \ept \left( S_{N,k}(s) \right).
\end{align*}
Therefore the estimate for $I_{1,1,2}$ follows
\begin{align*}
I_{1,1,2} \leq &  C \Vert \nabla \Phi^\varepsilon \Vert_{L^\infty(\R^2)}^{2k}  N^{-k_1+1} \Vert \nabla \Phi^\varepsilon \Vert_{L^\infty(\R^2)}^{2{k_1}} +  C \int_0^{t}  \ept \left( S_{N,k}(s  ) \right) ds.
\end{align*}

In the next step, we estimate $I_{1,2}$, which is the most difficult term in the whole argument. The key idea is to split the summation in $j$ by using Hölder's inequality on $\R^N$ for $(2k,\frac{2k}{2k-1})$,
\begin{align*}
& \displaystyle\sum_{j=1}^N \left| \nabla \Phi^\varepsilon((\bar{X}_{i}^\varepsilon - \bar{X}^\varepsilon_{j})(s  )) \right| \left| (X^\varepsilon_{N,j} - \bar{X}^\varepsilon_{j})(s  ) \right| \\
\leq & \bigg( \displaystyle\sum_{j=1}^N \left| \nabla \Phi^\varepsilon((\bar{X}_{i}^\varepsilon - \bar{X}^\varepsilon_{j})(s  )) \right|^{\frac{2k}{2k-1}} \bigg)^{\frac{2k-1}{2k}} \bigg(\displaystyle\sum_{j=1}^N \left| (X^\varepsilon_{N,j} - \bar{X}^\varepsilon_{j})(s  ) \right|^{2k} \bigg)^{\frac{1}{2k}}.
\end{align*}
This implies that the estimate for $I_{1,2}$ can be spitted into
\begin{align*}
 & I_{1,2} \\
 \leq &  C N^{2\alpha k} \frac{1}{N} \sum_{i=1}^N \ept \bigg( \int_0^{t \wedge \tau_{N,k}} \frac{1}{N^{2k}} \Big( \displaystyle\sum_{j=1}^N \left| \nabla \Phi^\varepsilon((\bar{X}_{i}^\varepsilon - \bar{X}^\varepsilon_{j})(s  )) \right|^{\frac{2k}{2k-1}} \Big)^{2k-1} \displaystyle\sum_{j=1}^N \left| (X^\varepsilon_{N,j} - \bar{X}^\varepsilon_{j})(s  ) \right|^{2k} ds \bigg)  \\
 \leq &  C \frac{1}{N} \sum_{i=1}^N \ept \bigg( \int_0^{t} \frac{1}{N^{2k-1}} \Big( \displaystyle\sum_{j=1}^N \left| \nabla \Phi^\varepsilon((\bar{X}_{i}^\varepsilon - \bar{X}^\varepsilon_{j})(s  )) \right|^{\frac{2k}{2k-1}} \Big)^{2k-1} N^{2\alpha k} \frac{1}{N} \displaystyle\sum_{j=1}^N \left| (X^\varepsilon_{N,j} - \bar{X}^\varepsilon_{j})(s \wedge \tau_{N,k}  ) \right|^{2k} ds \bigg) \\
 = &  C \frac{1}{N} \sum_{i=1}^N \ept \bigg( \int_0^{t} \Big( \frac{1}{N} \displaystyle\sum_{j=1}^N \left| \nabla \Phi^\varepsilon((\bar{X}_{i}^\varepsilon - \bar{X}^\varepsilon_{j})(s  )) \right|^{\frac{2k}{2k-1}} \Big)^{2k-1} S_{N,k}(s  ) ds \bigg) \\
 \leq &  C\frac{1}{N} \sum_{i=1}^N \ept \bigg( \int_0^{t} \Big( \left| \nabla \Phi^\varepsilon \right|^{\frac{2k}{2k-1}} \ast u^\varepsilon(s) (\bar{X}_i^\varepsilon (s  )) \Big)^{2k-1} S_{N,k}(s  ) ds \bigg) \\
  & +  C \frac{1}{N} \sum_{i=1}^N \ept \bigg( \int_0^{t} \Big| \frac{1}{N} \displaystyle\sum_{j=1}^N \left| \nabla \Phi^\varepsilon((\bar{X}_{i}^\varepsilon - \bar{X}^\varepsilon_{j})(s  )) \right|^{\frac{2k}{2k-1}} - \left| \nabla \Phi^\varepsilon \right|^{\frac{2k}{2k-1}} \ast u^\varepsilon(s) (\bar{X}_i^\varepsilon (s  )) \Big|^{2k-1} S_{N,k}(s  ) ds \bigg) \\
  =: & I_{1,2,1} + I_{1,2,2}.
\end{align*}
The estimates for $I_{1,2,1}$ and $I_{1,2,2}$ follow basically similar ideas to those for $I_{1,1,1}$ and $I_{1,1,2}$, but need more efforts. For $ I_{1,2,1}$ we get
\begin{align*}
 I_{1,2,1}
\leq  C \left\Vert \left| \nabla \Phi^\varepsilon \right|^{\frac{2k}{2k-1}} \ast u^\varepsilon \right\Vert_{L^\infty((0,T) \times \R^2)}^{2k-1} \int_0^{t} \ept \left(  S_{N,k}(s  ) \right) ds .
\end{align*}

To estimate $I_{1,2,2}$, we introduce a shortcut for the difference of mean and convolution with power
$$ \text{DMC}_i^k(s  ):= \frac{1}{N} \displaystyle\sum_{j=1}^N \left| \nabla \Phi^\varepsilon((\bar{X}_{i}^\varepsilon - \bar{X}^\varepsilon_{j})(s  )) \right|^{\frac{2k}{2k-1}} - \left| \nabla \Phi^\varepsilon \right|^{\frac{2k}{2k-1}} \ast u^\varepsilon(s) (\bar{X}_i^\varepsilon (s  ))$$
and $A^{N,k}_0(s):=A^N_0(|\nabla \Phi^\varepsilon|^{\frac{2k}{2k-1}}, u^\varepsilon(s))$ (see lemma \ref{lln}), then it follows that
\begin{align*}
I_{1,2,2} \leq &  C \int_0^{t} \ept \left( \max_{1 \leq i \leq N} \big| \text{DMC}_i^k(s  ) \big|^{2k-1} S_{N,k}(s  ) \right) ds =: C \int_0^{t}I_{1,2,2}'ds.
\end{align*}
Next, we split again the sample space $\Omega=A^{N,k}_{0}(s)\cup \big(A^{N,k}_{0}(s)\big)^c$. Due to the mass conservation of $u^\varepsilon$ mentioned earlier in the proof, it holds $\Vert \text{DMC}_i^k \Vert_{L^\infty((0,T) \times \R^2)} \leq 2 \Vert \nabla \Phi^\varepsilon \Vert_{L^\infty(\R^2)}^{\frac{2k}{2k-1}}$. Let $k_2 \in \N$. Using the law of large number result lemma \ref{lln} and \eqref{S_N_leq_1} we can show
\begin{align*}
I_{1,2,2}' =& \ept \Big( \mathbbm{1}_{A^{N,k}_{0}(s)} \max_{1 \leq i \leq N} \big| \text{DMC}_i^k(s  ) \big|^{2k-1} S_{N,k}(s  ) \Big) + \ept \Big( \mathbbm{1}_{(A^{N,k}_{0}(s))^c} \max_{1 \leq i \leq N} \big| \text{DMC}_i^k(s  ) \big|^{2k-1} S_{N,k}(s  ) \Big) \\
\leq & \ept \Big( \mathbbm{1}_{A^{N,k}_{0}(s)}  C \Vert \nabla \Phi^\varepsilon \Vert_{L^\infty(\R^2)}^{2k} S_{N,k}(s  ) \Big) + \ept \Big( \mathbbm{1}_{(A^{N,k}_{0}(s))^c}  S_{N,k}(s  ) \Big) \\
\leq  &  C \Vert \nabla \Phi^\varepsilon \Vert_{L^\infty(\R^2)}^{2k} \pr \left( A^{N,k}_{0}(s) \right) + \ept \left(  S_{N,k}(s  ) \right)\\
\leq  &  C \Vert \nabla \Phi^\varepsilon \Vert_{L^\infty(\R^2)}^{2k} N^{-k_2+1} \Vert \nabla \Phi^\varepsilon  \Vert_{L^\infty(\R^2)}^{\frac{2k}{2k-1}2{k_2}} + \ept \left(  S_{N,k}(s  ) \right),
\end{align*}
which implies the estimate for $I_{1,2,2}$:
\begin{align*}
I_{1,2,2} 
\leq &  C \Vert \nabla \Phi^\varepsilon \Vert_{L^\infty(\R^2)}^{2k}  N^{-k_2+1} \Vert \nabla \Phi^\varepsilon  \Vert_{L^\infty(\R^2)}^{\frac{2k}{2k-1}2{k_2}} +  C \int_0^{t} \ept \left(  S_{N,k}(s) \right) ds.
\end{align*}
The higher order term in the Taylor's formula, $I_2$, can be estimated by the cut-off parameter $N^{-\alpha}$, namely
\begin{align*}
I_2\leq &  C N^{2\alpha k} \left\Vert D^2 \Phi^\varepsilon \right\Vert_{L^\infty(\R^2)}^{2k} \frac{1}{N} \sum_{i=1}^N \ept \bigg( \int_0^{t \wedge \tau_{N,k}}\Big| \frac{1}{N}\displaystyle\sum_{j=1}^N \left( \left( \left| X_{N,i}^\varepsilon- \bar{X}_{i}^\varepsilon \right|(s  ) \right)^2 + \left( \left| X^\varepsilon_{N,j}- \bar{X}_{j}^\varepsilon \right|(s  ) \right)^2 \right) \Big|^{2k} ds \bigg) \\
\leq &  C N^{2\alpha k} \left\Vert D^2 \Phi^\varepsilon \right\Vert_{L^\infty(\R^2)}^{2k} \frac{1}{N} \sum_{i=1}^N \ept \bigg( \int_0^{t \wedge \tau_{N,k}} \left( \left| X_{N,i}^\varepsilon- \bar{X}_{i}^\varepsilon \right|(s  ) \right)^{4k} + \frac{1}{N}\displaystyle\sum_{j=1}^N \left( \left| X^\varepsilon_{N,j}- \bar{X}_{j}^\varepsilon \right|(s  ) \right)^{4k} \bigg) ds \\
\leq &  2 C N^{2\alpha k} \left\Vert D^2 \Phi^\varepsilon \right\Vert_{L^\infty(\R^2)}^{2k}  \int_0^{t} \ept \bigg(  \frac{1}{N} \sum_{i=1}^N \left( \left| X_{N,i}^\varepsilon- \bar{X}_{i}^\varepsilon \right|(s \wedge \tau_{N,k}  ) \right)^{4k} \bigg) ds.
\end{align*}
Since $N^{2\alpha k} \frac{1}{N} \sum\limits_{1 \leq i \leq N} \left| X^\varepsilon_{N,i} - \bar{X}^\varepsilon_{i} \right|^{2k}(t \wedge \tau_{N,k}) )=S_{N,k}(t ) \leq 1$ by \eqref{S_N_leq_1}, we conclude
\begin{align*}
& \left| X^\varepsilon_{N,i} - \bar{X}^\varepsilon_{i} \right|^{2k}(t \wedge \tau_{N,k} ) \leq \sum\limits_{1 \leq i \leq N} \left| X^\varepsilon_{N,i} - \bar{X}^\varepsilon_{i} \right|^{2k}(t \wedge \tau_{N,k} ) \leq N^{1-2 \alpha k},
\end{align*}
which implies
\begin{align*}
 &  \frac{1}{N} \sum_{i=1}^N \left( \left| X_{N,i}^\varepsilon- \bar{X}_{i}^\varepsilon \right|(s \wedge \tau_{N,k}  ) \right)^{4k} \leq \frac{1}{N} \sum_{i=1}^N \left( \left| X_{N,i}^\varepsilon- \bar{X}_{i}^\varepsilon \right|(s \wedge \tau_{N,k}  ) \right)^{2k}  N^{1-2 \alpha k}  = N^{-2\alpha k} S_{N,k}(t ) N^{1-2 \alpha k}.
\end{align*}
Therefore $I_2$ can be bounded by
\begin{align*}
I_2 \leq C \left\Vert D^2 \Phi^\varepsilon \right\Vert_{L^\infty(\R^2)}^{2k} N^{1-2 \alpha k} \int_0^{t} \ept \left(  S_{N,k}(s) \right) ds.
\end{align*}
The rest term is $II$, an error term which appears in the Gronwall argument, can be estimated directly by using the law of large numbers.
\begin{align*}
II
\leq &  C N^{2\alpha k} \int_0^{t}\ept \bigg( \max_{1 \leq i \leq N} \Big| \frac{1}{N}\displaystyle\sum_{j=1}^N\Phi^\varepsilon((\bar{X}_{i}^\varepsilon - \bar{X}^\varepsilon_{j})(s))  - \Phi^\varepsilon\ast u^\varepsilon(s, \bar{X}^\varepsilon_{i}(
s  )) \Big|^{2k} \bigg)=:C N^{2\alpha k}\int^t_0 II' ds.
\end{align*}
Let $\theta \in (0,\frac{1}{2} )$ and $k_3 \in \N$. We set $\tilde{A}^{N}_\theta(s):=A^{N}_\theta( \Phi^\varepsilon, u^\varepsilon(s))$ and use lemma \ref{lln} to conclude
\begin{align*}
II' \leq &  C \ept \bigg( \mathbbm{1}_{\tilde{A}^{N}_\theta(s)} \max_{1 \leq i \leq N} \Big| \frac{1}{N}\displaystyle\sum_{j=1}^N\Phi^\varepsilon((\bar{X}_{i}^\varepsilon - \bar{X}^\varepsilon_{j})(s  ))  - \Phi^\varepsilon\ast u^\varepsilon(s, \bar{X}^\varepsilon_{i}(
s  )) \Big|^{2k} \bigg) \\
& + C \ept \bigg( \mathbbm{1}_{(\tilde{A}^{N}_\theta(s))^c} \max_{1 \leq i \leq N} \Big| \frac{1}{N}\displaystyle\sum_{j=1}^N\Phi^\varepsilon((\bar{X}_{i}^\varepsilon - \bar{X}^\varepsilon_{j})(s  ))  - \Phi^\varepsilon\ast u^\varepsilon(s, \bar{X}^\varepsilon_{i}(
s  )) \Big|^{2k} \bigg)\\
 \leq &  C \Vert \Phi^\varepsilon \Vert_{L^\infty(\R^2)}^{2k} \pr \left(\tilde{A}^{N,k}_\theta(s) \right) + \frac{C(k)}{N^{2k\theta}}
\\
\leq &  C \Vert \Phi^\varepsilon \Vert_{L^\infty(\R^2)}^{2k}  N^{2{k_3}(\theta-\frac{1}{2})+1} \Vert \Phi^\varepsilon \Vert_{L^\infty(\R^2)}^{2{k_3}} + \frac{C}{N^{2k\theta}},
\end{align*}
which implies the estimate for $II$,
\begin{align*}
II \leq & N^{2\alpha k} \Vert \Phi^\varepsilon \Vert_{L^\infty(\R^2)}^{2k} N^{2{k_3}(\theta-\frac{1}{2})+1} \Vert \Phi^\varepsilon \Vert_{L^\infty(\R^2)}^{2{k_3}} +  C T N^{2\alpha k-2k\theta}.
\end{align*}
We combine all the estimations we did, and obtain
\begin{align*}
& \ept (S_{N,k}(t )) \\
\leq & I +II \leq I_1 + I_2 +II \leq I_{1,1} + I_{1,2} + I_2 + II \leq I_{1,1,1} + I_{1,1,2} + I_{1,2,1} + I_{1,2,2} +I_2 +II \\
\leq & C \bigg( \Vert \nabla \Phi^\varepsilon \Vert_{L^\infty(\R^2)}^{2k}  N^{-k_1+1} \Vert \nabla \Phi^\varepsilon \Vert_{L^\infty(\R^2)}^{2{k_1}} +  \Vert \nabla \Phi^\varepsilon \Vert_{L^\infty(\R^2)}^{2k}  N^{-k_2+1} \Vert \nabla \Phi^\varepsilon  \Vert_{L^\infty(\R^2)}^{\frac{2k}{2k-1}2{k_2}} \\
& +  \Vert \Phi^\varepsilon \Vert_{L^\infty(\R^2)}^{2k} N^{2\alpha k+2{k_3}(\theta-\frac{1}{2})+1} \Vert \Phi^\varepsilon \Vert_{L^\infty(\R^2)}^{2{k_3}} + N^{2(\alpha -\theta)k} \bigg) 
 +   C \Big( \Vert \left| \nabla \Phi^\varepsilon \right| \ast u^\varepsilon \Vert_{L^\infty((0,T) \times \R^2 )}^{2k}\\
  &\qquad+ 1 + \left\Vert \left| \nabla \Phi^\varepsilon \right|^{\frac{2k}{2k-1}} \ast u^\varepsilon \right\Vert_{L^\infty((0,T) \times \R^2 )}^{2k-1} + 1 + \left\Vert D^2 \Phi^\varepsilon \right\Vert_{L^\infty(\R^2)}^{2k} N^{1-2 \alpha k} \Big) \int_0^{t} \ept \left(  S_{N,k}(s ) \right) ds.
\end{align*}
By Gronwall's lemma we conclude $\forall t \in [0,T]$: 
\begin{align} \label{gronw_mean}
& \ept (S_{N,k}(t )) \notag \\
\leq &  C \bigg( \Vert \nabla \Phi^\varepsilon \Vert_{L^\infty(\R^2)}^{2k}  N^{-k_1+1} \Vert \nabla \Phi^\varepsilon \Vert_{L^\infty(\R^2)}^{2{k_1}} +  \Vert \nabla \Phi^\varepsilon \Vert_{L^\infty(\R^2)}^{2k}  N^{-k_2+1} \Vert \nabla \Phi^\varepsilon  \Vert_{L^\infty(\R^2)}^{\frac{2k}{2k-1}2{k_2}} \notag \\
& \qquad + \Vert \Phi^\varepsilon \Vert_{L^\infty(\R^2)}^{2k} N^{2\alpha k+2{k_3}(\theta-\frac{1}{2})+1} \Vert \Phi^\varepsilon \Vert_{L^\infty(\R^2)}^{2{k_3}} + N^{2(\alpha -\theta)k} \bigg) \notag \\
& \cdot \exp \bigg(  C \Big( \Vert \left| \nabla \Phi^\varepsilon \right| \ast u^\varepsilon \Vert_{L^\infty((0,T) \times \R^2 )}^{2k} + 1 + \left\Vert \left| \nabla \Phi^\varepsilon \right|^{\frac{2k}{2k-1}} \ast u^\varepsilon \right\Vert_{L^\infty((0,T) \times \R^2)}^{2k-1} + \left\Vert D^2 \Phi^\varepsilon \right\Vert_{L^\infty(\R^2)}^{2k} N^{1-2 \alpha k} \Big) T \bigg).
\end{align}

In the above inequality, we need to discuss the the asymptotics w.r.t $\varepsilon$ for those terms where $\Phi_\varepsilon$ is involved. To make the right-hand side decay to $0$ as $N$ goes to infinity, we then fix a range for $\varepsilon$ depending on $N$.

In the following we will use the properties \eqref{prop_Phi} of $\Phi$ and the well-known fact that $\Vert j^\varepsilon \Vert_{L^q(\R^2)} = \left( \frac{1}{\varepsilon} \right)^{2\frac{q-1}{q}} \Vert j \Vert_{L^q(\R^2)}$ and $\Vert \nabla j^\varepsilon \Vert_{L^q(\R^2)} = \left( \frac{1}{\varepsilon} \right)^{\frac{3q-2}{q}} \Vert \nabla j \Vert_{L^q(\R^2)}$ for all $q \in [1,\infty )$. Using Young's inequality, it is not hard to infer that
\begin{alignat}{4} \label{est_phi_eps}
 & \forall a>0 : \ & & \exists C>0 : \ & & \forall \varepsilon >0: & & \left\Vert \Phi^{\varepsilon} \right\Vert_{L^\infty(\R^2)} \leq \frac{C}{\varepsilon^a}, \\
 & \forall b>1 : \ & & \exists C>0 : \ & & \forall \varepsilon >0 : & & \left\Vert  \nabla \Phi^{\varepsilon} \right\Vert_{L^\infty(\R^2)} \leq  \frac{C}{\varepsilon^b}, \label{est_D_phi_eps} \\
 &\forall d>2 : \ & & \exists C>0 : \ & & \forall \varepsilon >0 : & & \left\Vert D^2 \Phi^{\varepsilon} \right\Vert_{L^\infty(\R^2)} \leq  \frac{C}{\varepsilon^d}.  \label{est_D^2_phi_eps}
\end{alignat}
The uniform bounds for $u^\varepsilon$ from lemma  \ref{bound_u^eps} lead to a bound for the remaining terms.
\begin{align}\label{nablaPhistaru}
\Vert \left| \nabla \Phi^\varepsilon \right| \ast u^\varepsilon \Vert_{L^\infty((0,T) \times \R^2)} \leq & C \Vert  \nabla \Phi \Vert_{L^{\frac{3}{2}}(\R^2 )} \Vert j^\varepsilon \Vert_{L^1(\R^2 )} \Vert u^\varepsilon \Vert_{L^\infty (0,T;L^3(\R^2))} \leq C
\end{align}

The estimation of the last term requires a slight restriction for $k$. \underline{From now on we assume} $k>1$. Then $\frac{2k}{2k-1}<2$. So we can choose some $p_k >1$ such that $\frac{2k}{2k-1} p_k <2$. Let $q_k$ be dual to $p_k$. Since $q_k < \infty$, we can use lemma \ref{bound_u^eps} in the following.
\begin{align} \label{bound_nabla_phi_eps_power_u_eps}
	& \left\Vert \left| \nabla \Phi^\varepsilon \right|^{\frac{2k}{2k-1}} \ast u^\varepsilon \right\Vert_{L^\infty((0,T)\times \R^2)} \leq \left\Vert \left| \nabla \Phi^\varepsilon \right|^{\frac{2k}{2k-1}} \right\Vert_{L^{p_k}(\R^2)} \left\Vert u^\varepsilon  \right\Vert_{L^\infty(0,T;L^{q_k}(\R^2))}
	\leq  C \left\Vert \nabla \Phi \right\Vert_{L^{\frac{2k}{2k-1}p_k}(\R^2)}^{\frac{2k}{2k-1}} \left\Vert j^\varepsilon \right\Vert_{L^1(\R^2)}^{\frac{2k}{2k-1}} \leq C
\end{align}
Let $a>0,b>1$ and $d>2$. From \eqref{gronw_mean} together with \eqref{est_phi_eps} - \eqref{bound_nabla_phi_eps_power_u_eps} we find for $\ept (S_{N,k}(t )) $
\begin{align*}
& \ept (S_{N,k}(t )) \\
\leq &  C\bigg(  \Big( \frac{C}{\varepsilon^b} \Big)^{2k}  N^{-{k_1}+1} \Big( \frac{C}{\varepsilon^b} \Big)^{2{k_1}} +  \Big( \frac{C}{\varepsilon^b} \Big)^{2k}  N^{-{k_2}+1} \Big( \frac{C}{\varepsilon^b} \Big)^{\frac{2k}{2k-1}2{k_2}} + \Big( \frac{C}{\varepsilon^a} \Big)^{2k} N^{2\alpha k+2{k_3}({\theta}-\frac{1}{2})+1} \Big( \frac{C}{\varepsilon^a} \Big)^{2{k_3}} + N^{2(\alpha -\theta)k} \bigg) \\
& \qquad\qquad\cdot \exp \bigg(  C\Big( C^{2k}+ 1 + C^{2k-1} + \Big( \frac{C}{\varepsilon^d} \Big)^{2k} N^{1-2 \alpha k} \Big) \bigg) \\
\leq & C \Big( \frac{1}{\varepsilon^{2(k+k_1)b}} N^{-{k_1}+1} + \frac{1}{\varepsilon^{2(k+\frac{2k}{2k-1}{k_2})b}} N^{-{k_2}+1} + \frac{1}{\varepsilon^{2(k+k_3)a}}  N^{2\alpha k+2{k_3}({\theta}-\frac{1}{2})+1} + N^{2(\alpha -\theta)k} \Big) \\
& \qquad\qquad\cdot \exp \bigg( C \Big( 1 + \frac{1}{\varepsilon^{2kd}} N^{1-2 \alpha k} \Big) \bigg) .
\end{align*}
Next we include the scaling of $\varepsilon$ w.r.t $N$. Let $\beta >0$ and $\gamma >0$. \underline{From now on we assume} $\varepsilon \geq N^{-\beta}$. Consequently, we obtain
\begin{align} \label{E(S_N_k)_N-bound}
 \ept (S_{N,k}(t )) \leq &  C \Big( N^{2(k+k_1)b\beta -{k_1}+1} + N^{2(k+\frac{2k}{2k-1}{k_2})b\beta -{k_2}+1} + N^{2(k+k_3)a\beta +2\alpha k+2{k_3}({\theta}-\frac{1}{2})+1} + N^{2(\alpha -\theta)k} \Big) \notag \\
&\qquad \qquad \cdot \exp \Big( C \left( 1 + N^{2kd\beta +1-2 \alpha k} \right) \Big) \notag \\
= & C \Big( N^{2kb \beta + (2b \beta  -1)k_1+1} + N^{2kb \beta + \left( \frac{4k}{2k-1} b\beta -1 \right) k_2 +1} + N^{2(a\beta + \alpha )k + 2 \left( a\beta + \theta - \frac{1}{2} \right) k_3 +1} + N^{2(\alpha -\theta)k} \Big) \notag \\
&\qquad \qquad \cdot \exp \Big( C \left( 1 + N^{1+ 2(d \beta - \alpha )k} \right) \Big).
\end{align}
Here we have already sorted by $k,k_1,k_2$ and $k_3$ within the exponents. Now it remains to select the parameters so that the coefficients are negative since they can be amplified by choosing $k,k_1,k_2$ and $k_3$ large. To achieve this, we will gradually add constraints for the parameters in the following.\\

To avoid any contribution of the exponential term, we need $N^{1+ 2(d \beta - \alpha )k}$ to be bounded. We see
\begin{align*}
d \beta - \alpha <0 \Leftrightarrow \beta < \frac{\alpha}{d}.
\end{align*}
\underline{From now on we assume} $\beta < \frac{\alpha}{d} \left( <\frac{\alpha}{2} \right)$. Then it holds
\begin{align} \label{cond_k_alpha_beta}
 1+ 2(d \beta - \alpha )k \leq 0 \Leftrightarrow k \geq - \frac{1}{2d\beta - 2\alpha}.
\end{align}
We want all the terms in the bracket preceding the exponential function to decay like $N^{-\gamma }$. Therefore, we go through them step by step.
\begin{enumerate}
\item[$\cdot$)] $ 2b \beta  -1 <0 \Leftrightarrow \beta < \frac{1}{2b} $.
\underline{From now on we assume} $\beta < \frac{1}{2b} \left( < \frac{1}{2} \right)$. This permits
\begin{align} \label{cond_k1}
2kb \beta + (2b \beta  -1)k_1+1 \leq -\gamma \Leftrightarrow k_1 \geq - \frac{\gamma+2kb \beta +1}{2b \beta -1}. 
\end{align}
\item[$\cdot$)] $\frac{4k}{2k-1} b\beta -1 <0 \Leftrightarrow \frac{2k-1}{4k}>b\beta \Leftrightarrow k > \frac{1}{2(1-2b \beta)} \text{ (note $2b \beta <1$)}$.
\underline{From now on we assume} $k > \frac{1}{2(1-2b \beta)}$. Hence it holds
\begin{align}
2kb \beta + \left( \frac{4k}{2k-1} b\beta -1 \right) k_2 +1 \leq -\gamma \Leftrightarrow k_2 \geq - \frac{\gamma + 2kb \beta +1}{\frac{4k}{2k-1} b\beta -1}.
\end{align}
\item[$\cdot$)] $a\beta + \theta - \frac{1}{2} <0 \Leftrightarrow \beta <\frac{1}{a} \left(\frac{1}{2} - \theta \right) (>0, \text{ since } \theta <\frac{1}{2}) $. \underline{From now on we assume} $\beta < \frac{1}{a} \left(\frac{1}{2} - \theta \right)$. We see
\begin{align} \label{cond_k3}
2(a\beta + \alpha )k + 2 \left( a\beta + \theta - \frac{1}{2} \right) k_3 +1 \leq - \gamma \Leftrightarrow k_3 \geq - \frac{\gamma+2( a \beta - \alpha )k+1}{2 \left( a\beta + \theta - \frac{1}{2} \right)} 
\end{align}
\item[$\cdot$)] \underline{From now on we assume} $\alpha < \theta$. Then $\alpha - \theta <0$ and
\begin{align} \label{cond_k_alpha_theta}
	2(\alpha -\theta)k \leq - \gamma \Leftrightarrow k \geq -\frac{\gamma}{2(\alpha - \theta )}.
\end{align}
\end{enumerate}
The above calculations yield the following: \\
It holds $\alpha < \theta < \frac{1}{2}$. By choosing $\theta=\frac{\alpha}{2}+\frac{1}{4} >0$, $\alpha$ can attain every value in $(0,\frac{1}{2})$, since our assumption $\alpha < \theta < \frac{1}{2}$ is always satisfied. We know that $\beta < \frac{\alpha}{2}$. Now we will argue that we can choose $a,b$ and $d$ such that $\beta$ can attain each value in $(0,\frac{\alpha}{2})$. Since $\frac{\alpha}{\beta} >2$ the choice $d=\frac{\alpha}{2\beta}+1 >2$ guarantees that $\beta < \frac{\alpha}{d}$ is satisfied. Notice that $\beta < \frac{1}{4}$. So we can set $b=\frac{1}{2}+\frac{1}{4\beta}>1$ and $\beta < \frac{1}{2b}$ is satisfied. Finally, since $\frac{\frac{1}{2}- \theta}{\beta}>0$ we can choose $a>0$ such that $\beta < \frac{1}{a} \left(\frac{1}{2} - \theta \right)$. The minimal condition for $k$ is not of our interest, so we impose a stricter condition on $k$ that can be written more concisely. \underline{From now on we assume} $k> \frac{1}{\alpha-2\beta} \vee \frac{2\gamma}{1-2\alpha}$. As we will clarify, this also fulfills all previous assumptions for $k$. First observe that $k \geq \frac{1}{\alpha-2\beta} =  \frac{1}{2\alpha -2d\beta}$ which implies the left-hand side of \eqref{cond_k_alpha_beta}. Furthermore, $\alpha <1$ implies $\frac{1}{\alpha-2\beta} > \frac{1}{1-2\beta}= \frac{1}{2(1-2b \beta)}$, which means that the assumption $k> \frac{1}{2(1-2b \beta)}$ is already fulfilled. Thereby, also the assumption $k>1$ is fulfilled, since $\beta >0$. It follows from $k > \frac{2\gamma}{1-2\alpha} =  \frac{\gamma}{2(\theta - \alpha )}$ that the left-hand side of \eqref{cond_k_alpha_theta} holds. Now we can choose $k_1,k_2$ and $k_3$ such that the left-hand sides of \eqref{cond_k1} - \eqref{cond_k3} hold. Finally we can conclude from \eqref{E(S_N_k)_N-bound} that
\begin{align*}
& \ept (S_{N,k}(t )) \leq C N^{-\gamma}.
\end{align*} 
\end{proof}

\section{Equivalence of convergence in power-mean and maximum}
In this section we show a general statement about the equivalence of the convergence of the k-mean and the maximum in the sense of probability. Afterwards, we will briefly mention how theorem \ref{ks_conv_of_max_X-bar_X} can be obtained from theorem \ref{conv_of_mean}.

\begin{lemma} \label{equi_mean_max}
Let $T \in (0,\infty )$ and $d \in \N$. We consider two families of stochastic processes \\ $Y^\varepsilon_{N,i} : [0,T] \times \Omega \to \R^d$ for $N\in \N, i=1,\dots ,N, \varepsilon >0 $ and $\bar{Y}^\varepsilon_{i} : [0,T] \times \Omega \to \R^d$ for $i \in \N, \varepsilon >0 $. \\
Let $\alpha_0 \in (0, \infty ]$ and $\Lambda : (0, \alpha_0 ) \to (0,\infty ]$ monotonic increasing. Then the following statements are equivalent.
\begin{itemize}
\item[a)] $ \forall \alpha \in (0, \alpha_0 ) : \forall \beta \in (0, \Lambda (\alpha )) :\forall \gamma >0 : \exists k_0 \in (0,\infty ) : \forall k \in (k_0 , \infty ) : \exists C>0 : \forall N \in \N, \varepsilon >0 \text{ with } \varepsilon \geq N^{-\beta} : 
$\begin{align*}
& \sup_{t \in [0,T]} \pr \bigg( \Big\{ \Big( \frac{1}{N} \sum_{i=1}^N \left| Y^\varepsilon_{N,i} - \bar{Y}^\varepsilon_{i} \right|^{k}(t)  \Big)^{\frac{1}{k}} \geq N^{-\alpha} \Big\} \bigg) \leq  \frac{C}{N^\gamma}.
\end{align*}
\item[b)] $ \forall \alpha \in (0, \alpha_0 ) : \forall \beta \in (0, \Lambda (\alpha )) :\forall \gamma >0 : \exists C>0 : \forall N \in \N, \varepsilon >0 \text{ with } \varepsilon \geq N^{-\beta} :$
\begin{align*}
& \sup_{t \in [0,T]} \pr \bigg( \Big\{ \max_{1 \leq i \leq N}  \big| Y^\varepsilon_{N,i} - \bar{Y}^\varepsilon_{i} \big|(t) \geq N^{-\alpha} \Big\} \bigg) \leq  \frac{C}{N^\gamma}.
\end{align*}
\item[c)] $\forall \alpha \in (0, \alpha_0 ) : \forall \beta \in (0, \Lambda (\alpha )) :\forall \gamma >0 : \exists C>0 : \forall k \in (0 , \infty ) : \forall N \in \N, \varepsilon >0 \text{ with } \varepsilon \geq N^{-\beta} : $ 
\begin{align*}
& \sup_{t \in [0,T]} \pr \bigg( \Big\{ \Big( \frac{1}{N} \sum_{i=1}^N \left| Y^\varepsilon_{N,i} - \bar{Y}^\varepsilon_{i} \right|^{k}(t)  \Big)^{\frac{1}{k}} \geq N^{-\alpha} \Big\} \bigg) \leq  \frac{C}{N^\gamma}.
\end{align*}
\end{itemize}
\end{lemma}

\begin{remark} By considering stochastic processes that are constant in time, we obtain a version of the above lemma for random variables. However, for our purpose, the original form of the lemma is more suitable, since we need to maintain the time independence of the constant $C$. The introduction of the function $\Lambda$ enables us to keep the explicit dependence of $\beta$ on $\alpha$, when switching between mean and maximum. The statement would still be true in a more general form, if we replace $``\forall \beta \in (0, \Lambda (\alpha ))"$ by $``\exists \beta_0>0 : \forall \beta \in (0, \beta_0 )"$. However, in this form we loose the information about the range of $\beta$. In view of $c)$ we see that e.g. the arithmetic mean is included by choosing $k=1$. There is a natural way to extend $c)$ to the case $k=0$ and so allowing $``k \in [0,\infty )"$. Since  $\Big( \frac{1}{N} \sum\limits_{i=1}^N \left| x_{i} \right|^{k} \Big)^{\frac{1}{k}}$ converges to $\Big(  \prod\limits_{i=1}^N \left| x_{i} \right| \Big)^{\frac{1}{N}}$ the geometric mean, for $k \to 0$, it provides a canonical choice. One can argue in the same way as in the proof of $``b)\Rightarrow c)"$ (note $\Big( \frac{1}{N} \sum\limits_{1 \leq i \leq N} \left| x_{i} \right|^{k} \Big)^{\frac{1}{k}} \geq \Big(  \prod\limits_{1 \leq i \leq N} \left| x_{i} \right| \Big)^{\frac{1}{N}}, k>0$). So we also obtain the convergence of the geometric mean here.
\end{remark}

Next we prove lemma \ref{equi_mean_max}. The step $``a) \Rightarrow b)"$ is  clearly the essential part of the argument.

\begin{proof} ``$a) \Rightarrow b):$ '' 
	
	Let $\alpha \in (0, \alpha_0 )$ and further $\beta \in (0, \Lambda (\alpha ))$, $\gamma >0$. For all $t \in [0,T], k>0, \varepsilon >0, N \in \N$ it holds
\begin{align*}
 \Big\{ \max_{1 \leq i \leq N}  \left| Y^\varepsilon_{N,i} - \bar{Y}^\varepsilon_{i} \right|(t) \geq N^{-\alpha} \Big\} \subset \bigg\{ \Big( \frac{1}{N} \sum_{i=1}^N \left| Y^\varepsilon_{N,i} - \bar{Y}^\varepsilon_{i} \right|^{k}(t)  \Big)^{\frac{1}{k}} \geq N^{-(\alpha + \frac{1}{k})} \bigg\},
\end{align*}
since $ \max\limits_{1 \leq i \leq N}  \left| Y^\varepsilon_{N,i} - \bar{Y}^\varepsilon_{i} \right|(t) \leq \Big( \sum\limits_{i=1}^N \left| Y^\varepsilon_{N,i} - \bar{Y}^\varepsilon_{i} \right|^{k}(t)  \Big)^{\frac{1}{k}} =  \Big( \frac{1}{N} \sum\limits_{i=1}^N \left| Y^\varepsilon_{N,i} - \bar{Y}^\varepsilon_{i} \right|^{k}(t)  \Big)^{\frac{1}{k}} N^{\frac{1}{k}}$. \\

Now we choose $\tilde{\alpha}$ such that $\alpha < \tilde{\alpha} < \alpha_0$. Since $\Lambda$ is monotonic increasing, we have $\beta \in (0,\Lambda (\tilde{\alpha}))$. By using the assumption $a)$, we can choose some $k >0$ such that there exists some $C>0$  such that:
\begin{align*}
\forall N \in \N, \varepsilon >0 \text{ with } \varepsilon \geq N^{-\beta} : \sup_{t \in [0,T]} \pr \bigg( \Big\{ \Big( \frac{1}{N}\sum\limits_{i=1}^N \left| Y^\varepsilon_{N,i} - \bar{Y}^\varepsilon_{i} \right|^{k}(t)  \Big)^{\frac{1}{k}} \geq N^{-\tilde{\alpha}} \Big\} \bigg) \leq  \frac{C}{N^\gamma},
\end{align*}
and in addition $\alpha < \alpha + \frac{1}{k} < \tilde{\alpha}$. Finally we can conclude for all  $N \in \N$ and $\varepsilon >0 \text{ with } \varepsilon \geq N^{-\beta}$:
\begin{align*}
& \sup_{t \in [0,T]} \pr \Big( \Big\{ \max_{1 \leq i \leq N}  \left| Y^\varepsilon_{N,i} - \bar{Y}^\varepsilon_{i} \right|(t) \geq N^{-\alpha} \Big\} \Big) \leq \sup_{t \in [0,T]} \pr \bigg( \Big\{ \Big( \frac{1}{N} \sum_{i=1}^N \left| Y^\varepsilon_{N,i} - \bar{Y}^\varepsilon_{i} \right|^{k}(t)  \Big)^{\frac{1}{k}} \geq N^{-(\alpha + \frac{1}{k})} \Big\} \bigg) \\
\leq & \sup_{t \in [0,T]} \pr \bigg( \Big\{ \Big( \frac{1}{N} \sum_{i=1}^N \left| Y^\varepsilon_{N,i} - \bar{Y}^\varepsilon_{i} \right|^{k}(t)  \Big)^{\frac{1}{k}} \geq N^{-\tilde{\alpha}} \Big\} \bigg) \leq \frac{C}{N^\gamma}.
\end{align*}
``$b) \Rightarrow c):$''

 Let $\alpha \in (0, \alpha_0 )$ and further $\beta \in (0, \Lambda (\alpha ))$, $\gamma >0$. For all $t \in [0,T], k>0, \varepsilon >0, N \in \N$ it holds
\begin{align*}
 \bigg\{ \Big( \frac{1}{N} \sum\limits_{i=1}^N \left| Y^\varepsilon_{N,i} - \bar{Y}^\varepsilon_{i} \right|^{k}(t)  \Big)^{\frac{1}{k}} \geq N^{-\alpha} \bigg\} \subset \Big\{ \max_{1 \leq i \leq N}  \left| Y^\varepsilon_{N,i} - \bar{Y}^\varepsilon_{i} \right|(t) \geq N^{-\alpha} \Big\} ,
\end{align*}
since $\Big( \frac{1}{N} \sum\limits_{i=1}^N \left| Y^\varepsilon_{N,i} - \bar{Y}^\varepsilon_{i} \right|^{k}(t)  \Big)^{\frac{1}{k}} \leq \Big( \frac{1}{N} \sum\limits_{i=1}^N \max\limits_{1 \leq i \leq N} \left| Y^\varepsilon_{N,i} - \bar{Y}^\varepsilon_{i} \right|^{k}(t)  \Big)^{\frac{1}{k}} = \max\limits_{1 \leq i \leq N}  \left| Y^\varepsilon_{N,i} - \bar{Y}^\varepsilon_{i} \right|(t)$. \\
In view of assumption $b)$ there exists some $C>0$ such that
\begin{align*}
\forall N \in \N, \varepsilon >0 \text{ with } \varepsilon \geq N^{-\beta} : 
\sup_{t \in [0,T]} \pr \left( \left\lbrace \max_{1 \leq i \leq N}  \left| Y^\varepsilon_{N,i} - \bar{Y}^\varepsilon_{i} \right|(t) \geq N^{-\alpha} \right\rbrace \right) \leq  \frac{C}{N^\gamma}.
\end{align*}
Putting everything together we find for all $k>0$ and for all $N \in \N$ and $\varepsilon >0 \text{ with } \varepsilon \geq N^{-\beta}$:
\begin{align*}
\sup_{t \in [0,T]} \pr \bigg( \Big\{ \Big( \frac{1}{N} \sum_{i=1}^N \left| Y^\varepsilon_{N,i} - \bar{Y}^\varepsilon_{i} \right|^{k}(t)  \Big)^{\frac{1}{k}} \geq N^{-\alpha} \Big\}  \bigg) \leq \sup_{t \in [0,T]} \pr \Big(  \Big\{ \max_{1 \leq i \leq N}  \left| Y^\varepsilon_{N,i} - \bar{Y}^\varepsilon_{i} \right|(t) \geq N^{-\alpha} \Big\} \Big) \leq \frac{C}{N^\gamma}.
\end{align*}
``$c) \Rightarrow a):$'' Trivial.
\end{proof}

Now we can see that theorem \ref{ks_conv_of_max_X-bar_X} follows directly from theorem \ref{conv_of_mean} and lemma \ref{equi_mean_max} by choosing $\alpha_0:=\frac{1}{2}$ and $\Lambda (\alpha ):=\frac{\alpha}{2}$.

\section{Proof of theorem \ref{ks_conv_of_max_X-hat_X} and \ref{poc_strong}}
This section discuss the step from $\bar{X}^\varepsilon_{i}$ to $\hat{X}_{i}$. Then by combing this result with the convergence from theorem \ref{ks_conv_of_max_X-bar_X}, we obtain theorem \ref{ks_conv_of_max_X-hat_X}. Based on the latter, we also provide the strong propagation of chaos stated in theorem \ref{poc_strong} by using relative entropy methods. 

We start with a generalization of \cite[Proposition 15]{BOL2026113712}.

\begin{lemma} \label{conv_ept} Let $T >0$. Under the assumptions \ref{ass_ks} it holds:
$\forall k \in [1,\infty) : \exists  C>0 : \forall i \in \N,  \varepsilon >0 : $
\begin{align*}
\ept \Big( \sup_{t\in [0,T] } | \bar{X}^\varepsilon_i - \hat{X}_i |^{2k}(t) \Big) \leq  C \varepsilon^{2k}.
\end{align*}
\end{lemma}

\begin{proof} Let $T \in (0,\infty)$ and $ k \in [1,\infty)$. Further let
$t \in [0,T]$, $i \in \N$ and $\varepsilon >0$. In the whole proof $C$ denotes a generic constant which may vary in each step and which is independent of $t$, $\varepsilon$ or $i$. (So $C$ is allowed to depend on $k$.) As abbreviation we introduce
$$ \sigma_{i}(s) := \sqrt{2 \exp \big( -\Phi^\varepsilon \ast u^\varepsilon(s,\bar{X}^\varepsilon_i(s)) \big)+ 2  } - \sqrt{2 \exp \big(- \Phi \ast u(s,\hat{X}_i(s)) \big)+2  }.$$
First we use the Burkholder-Davis-Gundy inequality, which leads to
\begin{align*}
\ept \Big( \sup_{s\in [0,t] } | \bar{X}^\varepsilon_i - \hat{X}_i |^{2k}(s) \Big) 
= & \ept \bigg( \sup_{s\in [0,t] } \Big| \int_0^s  \sigma_{i} dB_i \Big|^{2k}  \bigg) \leq  C \ept \bigg( \Big( \int_0^t  \sigma_{i}^2 d s \Big)^k \bigg).
\end{align*}
By estimating the time integral with Hölder's inequality and taking into account that $\sqrt{2(\cdot )+2}$ and $\exp(-(\cdot ))$ are Lipschitz continuous, we obtain
\begin{align*}
 \ept \Big( \sup_{s\in [0,t] } | \bar{X}^\varepsilon_i - \hat{X}_i |^{2k}(s) \Big)  \leq  C \ept  \left( \int_0^t  \sigma_{i}^{2k} d s \right) 
\leq &  C \ept  \left( \int_0^t \big| \Phi^\varepsilon \ast u^\varepsilon(s,\bar{X}^\varepsilon_i(s  )) -   \Phi \ast u(s,\hat{X}_i(s  )) \big|^{2k} d s \right).
\end{align*}
Now we divide the integrand in three parts:
\begin{align*}
A_1 := &   \Phi^\varepsilon \ast u^\varepsilon(s,\bar{X}^\varepsilon_i)-\Phi^\varepsilon \ast u(s,\bar{X}^\varepsilon_i), \ A_2 := \Phi^\varepsilon \ast u(s,\bar{X}^\varepsilon_i)- \Phi^\varepsilon \ast u(s,\hat{X}_i), \ A_3 := \Phi^\varepsilon \ast u(s,\hat{X}_i)- \Phi \ast u(s,\hat{X}_i).
\end{align*}
We estimate accordingly using the convexity of $( \cdot )^{2k}$.
\begin{align} \label{A5r}
&  \ept \Big( \sup_{s\in [0,t] } | \bar{X}^\varepsilon_i - \hat{X}_i |^{2k}(s) \Big)   \leq   C(k) \ept \bigg( \int_0^{t} A_1^{2k} +  A_2^{2k} + A_3^{2k} d s \bigg).
\end{align}
We discuss the terms separately starting with $A_1$. A transformation yields
\begin{align}
\nonumber\ept \left(\int_0^t A_1^{2k}(s)\,ds \right) = & \int_0^t \int_{\R^2} \left( \Phi^\varepsilon \ast u^\varepsilon(s,x) -  \Phi^\varepsilon \ast u(s,x) \right)^{2k} u^\varepsilon(s,x) \,dx d s \\
\nonumber\leq & T \|\Phi^\varepsilon \ast (u^\varepsilon-u)\|_{L^\infty(0,T;L^{4k}(\mathbb{R}^2))}^{2k} \|u^\varepsilon\|_{L^\infty(0,T;L^2(\mathbb{R}^2))}  \\
\nonumber \leq &  C \|\Phi^\varepsilon \|_{L^{\frac{4k}{2k+1}}(\mathbb{R}^2)}^{2k} \| (u^\varepsilon-u) \|_{L^\infty(0,T;L^{2}(\mathbb{R}^2))}^{2k} \|u^\varepsilon\|_{L^\infty(0,T;L^2(\mathbb{R}^2))} \\
\label{A6r} \leq &  C \|\Phi \|_{L^{\frac{4k}{2k+1}}(\mathbb{R}^2)}^{2k} \|j^\varepsilon \|_{L^1(\mathbb{R}^2)}^{2k} \| (u^\varepsilon-u) \|_{L^\infty(0,T;L^{2}(\mathbb{R}^2))}^{2k} \|u^\varepsilon\|_{L^\infty(0,T;L^2(\mathbb{R}^2))} \leq C \varepsilon^{2k} ,
\end{align}
where in the last step we used the uniform bound from lemma \ref{bound_u^eps} and the error estimate $\|u^\varepsilon- u\|_{L^{\infty} (0,T; L^2 (\R^2) ) } \leq  C \varepsilon$ from \cite[Theorem 2]{BOL2026113712}.

To treat $A_2$, we mention that $u \in L^q(\R^2)$ for all $q < \infty$, as stated in \cite[Theorem 2]{BOL2026113712}. Therefore we can conclude
\begin{align}\label{A7r}
\ept \left(\int_0^t A_2^{2k}(s)\,ds \right)
\leq  & \|\nabla\Phi^\varepsilon\ast u\|_{L^\infty((0,T)\times\mathbb{R}^2)}^ {2k}
 \ept \left( \int_0^t |\bar{X}^\varepsilon_i- \hat{X}_i  |^{2k}(s) \,d s \right)
\leq  C(k)\int_0^t \ept \left( |\bar{X}^\varepsilon_i- \hat{X}_i  |^{2k}(s) \right) \,d s.
\end{align}
To estimate the last term we argue similarly to $A_1$, keeping in mind that $u \in L^2(\R^2)$.
\begin{align*}
\ept \left(\int_0^t A_3^{2k}(s)\,ds \right) 
 &= \int_0^t \int_{\R^2} \left( \Phi^\varepsilon \ast u(s,x) -  \Phi\ast u(s,x) \right)^{2k} u(x,s) \, dx ds \\
  &\leq c\|\Phi^\varepsilon-\Phi\|_{L^{\frac{4k}{2k+1}}(\mathbb{R}^2)}^{2k}\|u\|_{L^\infty(0,T;L^2(\mathbb{R}^2))}^{2k+1}. 
\end{align*}
Since $\frac{4k}{2k+1} < 2$, it holds $\nabla \Phi \in L^{\frac{4k}{2k+1}}(\mathbb{R}^2)$ by \eqref{prop_Phi} and by  \cite[Lemma 13]{BOL2026113712} we get $\|\Phi \ast j^\varepsilon -\Phi\|_{L^{\frac{4k}{2k+1}}(\mathbb{R}^2)} \leq  \| \nabla \Phi \|_{L^{\frac{4k}{2k+1}}(\mathbb{R}^2)} \varepsilon $. This leads to
\begin{align}\label{A8r}
\ept \Big(\int_0^t A_3^{2k}(s)\,ds \Big) \leq C \| \nabla \Phi \|_{L^{\frac{4k}{2k+1}}(\mathbb{R}^2)}^{2k} \varepsilon^{2k} \leq  C \varepsilon^{2k}.
\end{align}
In view of \eqref{A5r} the estimations \eqref{A6r}, \eqref{A7r} and \eqref{A8r} imply
$$ \ept \Big( \sup_{s\in [0,t] } | \bar{X}^\varepsilon_i - \hat{X}_i |^{2k}(s) \Big) \leq  C \varepsilon^{2k} +  C \int_0^t \ept \Big( \sup_{r\in [0,s] } |\bar{X}^\varepsilon_i- \hat{X}_i |^{2k}(r)\Big) \,d s \ \forall t \in [0,T] . $$ 
Using Gronwall's lemma we conclude
$$ \ept \Big( \sup_{s\in [0,T] } | \bar{X}^\varepsilon_i - \hat{X}_i |^{2k}(s) \Big) \leq  C \varepsilon^{2k} .$$
\end{proof}

Then combining theorem \ref{ks_conv_of_max_X-bar_X} and lemma \ref{conv_ept}, we can prove theorem \ref{ks_conv_of_max_X-hat_X} in the following:

\begin{proof}[Proof of theorem \ref{ks_conv_of_max_X-hat_X}]

Let $T \in (0,\infty )$, $\beta \in (0, \frac{1}{4}) $ and $\gamma >0$. Further let $t \in [0,T]$, $2\leq N\in \N$. We set $\varepsilon := N^{- \beta}$. In the whole proof $C$ denotes a generic  constant which may vary in each step and which is independent of $t$ or $N$ (and hence $\varepsilon$). (So $C$ is allowed to depend on $k$, which soon will appear in the proof.) \\

\noindent
First we compare $ X^\varepsilon_{N,i}$ and $\hat{X}_{i}$ to $\bar{X}^\varepsilon_{i}$.
\begin{align} \label{ks_conv_of_max_X-hat_X_pf_X-hat_X}
& \pr \Big( \Big\{ \max_{1 \leq i \leq N}  \left| X^\varepsilon_{N,i} - \hat{X}_{i} \right|(t) \geq \varepsilon^\eta \Big\}\Big) \leq \pr \Big( \Big\{ \max_{1 \leq i \leq N}  \left| X^\varepsilon_{N,i} - \bar{X}^\varepsilon_{i} \right|(t) + \max_{1 \leq i \leq N}  \left| \bar{X}^\varepsilon_{i} - \hat{X}_{i} \right|(t)\geq \varepsilon^\eta \Big\} \Big) \notag \\
\leq & \pr \Big( \Big\{ \max_{1 \leq i \leq N}  \left| X^\varepsilon_{N,i} - \bar{X}^\varepsilon_{i} \right|(t) \geq \varepsilon^{2\eta } \Big\} \Big) + \pr \Big( \Big\{ \max_{1 \leq i \leq N}  \left| \bar{X}^\varepsilon_{i} - \hat{X}_{i} \right|(t)\geq \varepsilon^\eta(1-\varepsilon^\eta ) \Big\} \Big).
\end{align}

\noindent
We start with the estimation of $\bar{X}^\varepsilon_{i} - \hat{X}_{i}$. Let $k \geq 2$. Since we want to use Markov's inequality, note that \\
 $\varepsilon = N^{- \beta} \leq 2^{- \beta } <1$. Together with lemma \ref{conv_ept} we obtain
\begin{align*}
& \pr \Big( \Big\{ \max_{1 \leq i \leq N}  \left| \bar{X}^\varepsilon_{i} - \hat{X}_{i} \right|(t)\geq \varepsilon^\eta(1-\varepsilon^\eta ) \Big\} \Big) = \pr \Big( \Big\{ \max_{1 \leq i \leq N}  \left| \bar{X}^\varepsilon_{i} - \hat{X}_{i} \right|^k(t)\geq \varepsilon^{\eta k}(1-\varepsilon^\eta )^k \Big\} \Big) \\
\leq & \pr \Big( \Big\{ \sum_{i=1}^N  \left| \bar{X}^\varepsilon_{i} - \hat{X}_{i} \right|^k(t)\geq \varepsilon^{\eta k}(1-\varepsilon^\eta )^k \Big\} \Big) \leq \varepsilon^{-\eta k}(1-\varepsilon^\eta )^{-k} \sum_{i=1}^N  \ept \left( \left| \bar{X}^\varepsilon_{i} - \hat{X}_{i} \right|^k(t) \right)\\
\leq & \varepsilon^{-\eta k}(1-\varepsilon^\eta )^{-k} \sum_{i=1}^N C \varepsilon^k =  C (1-\varepsilon^\eta )^{-k} \varepsilon^{(1-\eta )k} N = C (1-\varepsilon^\eta )^{-k} \varepsilon^{(1-\eta )k - \frac{1}{\beta}}.
\end{align*}
Because $\varepsilon \leq 2^{- \beta }$, the term $(1-\varepsilon^\eta )^{-k}$ is bounded w.r.t $\varepsilon$. By choosing $k \geq \frac{\gamma + \frac{1}{\beta}}{1-\eta} $ such that $(1-\eta )k - \frac{1}{\beta} \geq \gamma$, we find
\begin{align} \label{ks_conv_of_max_X-hat_X_pf_bar_X-hat_X}
\pr \Big( \Big\{ \max_{1 \leq i \leq N}  \left| \bar{X}^\varepsilon_{i} - \hat{X}_{i} \right|(t)\geq \varepsilon^\eta(1-\varepsilon^\eta ) \Big\} \Big) \leq  C \varepsilon^\gamma .
\end{align}

\noindent
Next we will discuss the estimation of $X^\varepsilon_{N,i} - \bar{X}^\varepsilon_{i}$. Since $2 \beta < \frac{1}{2}$, we can choose $\alpha >0$ such that $2 \beta < \alpha < \frac{1}{2}$. By theorem \ref{ks_conv_of_max_X-bar_X} it holds
\begin{align*}
\sup_{t \in [0,T]} \pr \Big( \Big\{ \max_{1 \leq i \leq N}  \left| X^\varepsilon_{N,i} - \bar{X}^\varepsilon_{i} \right|(t) \geq N^{-\alpha} \Big\} \Big) \leq  \frac{C}{N^{\beta \gamma }} .
\end{align*}
Since $\varepsilon^{2\eta } = N^{-2 \beta \eta } \geq N^{- \alpha } $, we can conclude
\begin{align} \label{ks_conv_of_max_X-hat_X_pf_X-bar_X}
\pr \Big( \Big\{ \max_{1 \leq i \leq N} \left| X^\varepsilon_{N,i} - \bar{X}^\varepsilon_{i} \right|(t) \geq \varepsilon^{2\eta } \Big\} \Big) \leq \pr \Big( \Big\{ \max_{1 \leq i \leq N} \left| X^\varepsilon_{N,i} - \bar{X}^\varepsilon_{i} \right|(t) \geq N^{- \alpha} \Big\} \Big) \leq \frac{C}{N^{\beta \gamma }} =  C \varepsilon^\gamma .
\end{align}

\noindent
Combining \eqref{ks_conv_of_max_X-hat_X_pf_X-hat_X} with \eqref{ks_conv_of_max_X-hat_X_pf_X-bar_X} and \eqref{ks_conv_of_max_X-hat_X_pf_bar_X-hat_X}, we find
\begin{align*}
\pr \Big( \Big\{ \max_{1 \leq i \leq N}  \left| X^\varepsilon_{N,i} - \hat{X}_{i} \right|(t) \geq \varepsilon^\eta \Big\} \Big) \leq C \varepsilon^\gamma .
\end{align*}
\end{proof}

Now we prove theorem \ref{poc_strong}. 

\begin{proof}[Proof of theorem \ref{poc_strong}.] 
Let $T>0$,$\alpha \in (0, \frac{1}{2} )$, $\beta \in (0, \frac{\alpha}{2} )$, $\gamma >0$. Moreover let $\varepsilon >0$ and $N \in \N$ such that $\varepsilon \geq N^{-\beta}$ and $t \in [0,T]$. Whenever $i$ appears, it is considered an element of $\lbrace 1, \dots , N \rbrace$. We use again $C$ as a generic constant which does not depend on $\varepsilon ,N,t$ or $i$.

We follow the proof of \cite[Theorem 5]{BOL2026113712} and mention where we need to argue in a different way. We also repeat some steps for the readers convenience. \\

We start with the equation \cite[(93)]{BOL2026113712}:
\begin{align*}
\frac{d}{dt}\mathcal{H}(u_N^\varepsilon|{u^\varepsilon}^{\otimes N})
=&-\frac{1}{N}\sum_{i=1}^N\int_{\mathbb{R}^{2N}}u_N^\varepsilon\Big|\nabla_{x_i}\log \frac{u_N^\varepsilon}{{u^\varepsilon}^{\otimes N}} \Big|^2\,dx_1\cdots dx_N\nonumber\\
&-\frac{1}{N}\sum_{i=1}^N\int_{\mathbb{R}^{2N}}u_N^\varepsilon\exp\Big(-\frac{1}{N}\sum_{j=1}^N\Phi^\varepsilon(x_i-x_j) \Big)
\Big|\nabla_{x_i}\log\frac{u_N^\varepsilon}{{u^\varepsilon}^{\otimes N}} \Big|^2\,dx_1\cdots dx_N
+I_1+I_2,\end{align*}
in which
\begin{align*}
I_1:=&-\frac{1}{N}\sum_{i=1}^N\int_{\mathbb{R}^{2N}}u_N^\varepsilon\exp\Big(-\frac{1}{N}\sum_{j=1}^N\Phi^\varepsilon(x_i-x_j) \Big)\nonumber\\
&\phantom{xxxxxxx}\cdot\Big(\nabla_{x_i}\Phi^\varepsilon\ast u^\varepsilon(t,x_i)-\frac{1}{N}\sum_{j=1}^N\nabla_{x_i}\Phi^\varepsilon(x_i-x_j) \Big)
\cdot\nabla_{x_i}\log \frac{u_N^\varepsilon}{{u^\varepsilon}^{\otimes N}}\,dx_1\cdots dx_N
\end{align*}
and
\begin{align*}
I_2:=&-\frac{1}{N}\sum_{i=1}^N\int_{\mathbb{R}^{2N}}u_N^\varepsilon
\Big(\exp\Big(-\frac{1}{N}\sum_{j=1}^N\Phi^\varepsilon(x_i-x_j) \Big)-\exp(-\Phi^\varepsilon\ast u^\varepsilon(t,x_i)) \Big)\nonumber\\
&\phantom{xxxxxxx}\cdot \big(\nabla_{x_i}\log {u^\varepsilon}^{\otimes N}-\nabla_{x_i}\Phi^\varepsilon\ast u^\varepsilon (t,x_i)\big)
\cdot \nabla_{x_i}\log \frac{u_N^\varepsilon}{{u^\varepsilon}^{\otimes N}}\,dx_1\cdots dx_N.
\end{align*}
We divide the terms $I_1$ and $I_2$ in the same way as in \cite{BOL2026113712}.
\begin{align*}
I_1\leq& \frac{1}{2N}\sum_{i=1}^N\int_{\mathbb{R}^{2N}}u_N^\varepsilon\exp\Big(-\frac{1}{N}\sum_{i=1}^N\Phi^\varepsilon(x_i-x_j) \Big)
\Big|\nabla_{x_i}\log \frac{u_N^\varepsilon}{{u^\varepsilon}^{\otimes N}} \Big|^2\,dx_1\cdots dx_N\nonumber\\
&+C\frac{1}{N}\sum_{i=1}^N\mathbb{E}\big(\big|\nabla \Phi^\varepsilon\ast u^\varepsilon(t,X_{N,i}^{\varepsilon}(t ))
-\nabla \Phi^\varepsilon\ast u^\varepsilon(t,\bar{X}_i^\varepsilon(t  )) \big|^2 \big)\nonumber\\
&+C\frac{1}{N}\sum_{i=1}^N\mathbb{E}\Big(\Big|\nabla \Phi^\varepsilon\ast u^\varepsilon(t,\bar{X}_i^\varepsilon(t  ))
-\frac{1}{N}\sum_{j=1}^N\nabla \Phi^\varepsilon \left( (\bar{X}_i^\varepsilon-\bar{X}_j^\varepsilon)(t  ) \right) \Big|^2 \Big)\nonumber
\\
&+C\frac{1}{N}\sum_{i=1}^N\mathbb{E}\Big(\Big|\frac{1}{N}\sum_{j=1}^N\big|\nabla \Phi^\varepsilon \left( (\bar{X}_i^\varepsilon-\bar{X}_j^\varepsilon)(t  ) \right)
-\nabla \Phi^\varepsilon \left( (X_{N,i}^\varepsilon-X_{N,j}^\varepsilon)(t  ) \right) \big| \Big|^2 \Big)\nonumber\\
=:&\frac{1}{2N}\sum_{i=1}^N\int_{\mathbb{R}^{2N}}u_N^\varepsilon\exp\Big(-\frac{1}{N}\sum_{i=1}^N\Phi^\varepsilon(x_i-x_j) \Big)
\Big|\nabla_{x_i}\log \frac{u_N^\varepsilon}{{u^\varepsilon}^{\otimes N}} \Big|^2\,dx_1\cdots dx_N
+I_{11}+I_{12}+I_{13}.
\end{align*}
Now we come to the part, where the proof differs. We estimate the terms $I_{11}$ and $I_{13}$. \\

Therefore, we introduce the set $ A_N^\varepsilon(s):= \left\lbrace \max\limits_{1 \leq i \leq N} |X_{N,i}^{\varepsilon} - \bar{X}_i^\varepsilon |(s) \geq N^{-\alpha} \right\rbrace $. Since $
\|D^2\Phi^\varepsilon\ast u^\varepsilon\|_{L^\infty((0,T)\times\mathbb{R}^2)} \leq \frac{C}{\varepsilon} $
 (see proof of \cite[Theorem 5]{BOL2026113712}), we get
\begin{align*}
\ept \Big( \mathbbm{1}_{ {A_N^\varepsilon(t)}^c} \big|\nabla \Phi^\varepsilon\ast u^\varepsilon(t,X_{N,i}^{\varepsilon}(t ))
-\nabla \Phi^\varepsilon\ast u^\varepsilon(t,\bar{X}_i^\varepsilon(t  )) \big|^2 \Big) \leq  \ept \Big( \mathbbm{1}_{ {A_N^\varepsilon(t)}^c} \Big( \frac{C}{\varepsilon} \Big)^2 \big|X_{N,i}^{\varepsilon} - \bar{X}_i^\varepsilon|^2(t  ) \Big)  \leq \frac{C}{\varepsilon^2} N^{-2 \alpha }.
\end{align*}
From $\|\nabla \Phi^\varepsilon\ast u^\varepsilon\|_{L^\infty((0,T)\times\mathbb{R}^2)} \leq C$, as given in \eqref{nablaPhistaru} and theorem \ref{ks_conv_of_max_X-bar_X}, we obtain
\begin{align*}
\ept \Big( \mathbbm{1}_{ {A_N^\varepsilon(t)}} \big|\nabla \Phi^\varepsilon\ast u^\varepsilon(t,X_{N,i}^{\varepsilon}(t ))
-\nabla \Phi^\varepsilon\ast u^\varepsilon(t,\bar{X}_i^\varepsilon(t  )) \big|^2 \Big) \leq C \pr \big(A_N^\varepsilon(t)\big) \leq C N^{-\gamma }.
\end{align*}
Putting these estimations together, we find
\begin{align*}
 I_{11} =& C\frac{1}{N}\sum_{i=1}^N \ept \Big( \mathbbm{1}_{ {A_N^\varepsilon(t)}^c} \big|\nabla \Phi^\varepsilon\ast u^\varepsilon(t,X_{N,i}^{\varepsilon}(t ))
-\nabla \Phi^\varepsilon\ast u^\varepsilon(t,\bar{X}_i^\varepsilon(t  )) \big|^2 \Big) \\
& + C\frac{1}{N}\sum_{i=1}^N \ept \Big( \mathbbm{1}_{ {A_N^\varepsilon(t)}} \big|\nabla \Phi^\varepsilon\ast u^\varepsilon(t,X_{N,i}^{\varepsilon}(t  ))-\nabla \Phi^\varepsilon\ast u^\varepsilon(t,\bar{X}_i^\varepsilon(t  )) \big|^2 \Big) 
\leq  \frac{C}{\varepsilon^2} N^{-2 \alpha } + C N^{-\gamma }
\end{align*}
The term $\|D^2\Phi^\varepsilon\|_{L^\infty(\mathbb{R}^2)}$ is critical and needs a more careful treatment. So we just keep it for the time being.
\begin{align*}
& \mathbb{E}\Big(\mathbbm{1}_{ {A_N^\varepsilon(t)}^c} \Big|\frac{1}{N}\sum_{j=1}^N\big|\nabla \Phi^\varepsilon \left( (\bar{X}_i^\varepsilon-\bar{X}_j^\varepsilon )(t  ) \right) - \nabla \Phi^\varepsilon \left( (X_{N,i}^\varepsilon-X_{N,j}^\varepsilon)(t  ) \right)\big| \Big|^2 \Big)  \\
\leq & \mathbb{E}\Big(\mathbbm{1}_{ {A_N^\varepsilon(t)}^c} \Big|\frac{1}{N}\sum_{j=1}^N \|D^2\Phi^\varepsilon\|_{L^\infty(\mathbb{R}^2)} \big| (\bar{X}_i^\varepsilon -X_{N,i}^\varepsilon -(\bar{X}_j^\varepsilon-X_{N,j}^\varepsilon))(t) \big| \Big|^2 \Big) \leq  C\|D^2\Phi^\varepsilon\|_{L^\infty(\mathbb{R}^2)}^2 N^{-2\alpha }.
\end{align*}
Because $ \Vert \nabla \Phi^\varepsilon \Vert_{L^\infty(\mathbb{R}^2)} \leq \|\nabla \Phi \|_{L^1(\mathbb{R}^2)} \| j^\varepsilon\|_{L^{\infty}(\mathbb{R}^2)} \leq \frac{C}{\varepsilon^2}$, we conclude with the help of theorem \ref{ks_conv_of_max_X-bar_X}
\begin{align*}
& \mathbb{E}\bigg(\mathbbm{1}_{ {A_N^\varepsilon(t)}} \Big|\frac{1}{N}\sum_{j=1}^N\big|\nabla \Phi^\varepsilon \left( (\bar{X}_i^\varepsilon-\bar{X}_j^\varepsilon )(t  ) \right) - \nabla \Phi^\varepsilon \left( (X_{N,i}^\varepsilon-X_{N,j}^\varepsilon)(t  ) \right)\big| \Big|^2 \bigg)\leq \frac{C}{\varepsilon^4} \pr ( A_N^\varepsilon(t) ) \leq \frac{C}{\varepsilon^4} N^{-\gamma }.
\end{align*}
So we find for $I_{13}$:
\begin{align*}
I_{13} \leq &  C\|D^2\Phi^\varepsilon\|_{L^\infty(\mathbb{R}^2)}^2 N^{-2\alpha } + \frac{C}{\varepsilon^4} N^{-\gamma }.
\end{align*}
From \cite[(97)]{BOL2026113712} we know $I_{12} \leq \frac{C}{\varepsilon^4}N^{-1}$ and we finally receive
\begin{align*}
I_1 \leq & \frac{1}{2N}\sum_{i=1}^N\int_{\mathbb{R}^{2N}}u_N^\varepsilon\exp\Big(-\frac{1}{N}\sum_{i=1}^N\Phi^\varepsilon(x_i-x_j) \Big)
\Big|\nabla_{x_i}\log \frac{u_N^\varepsilon}{{u^\varepsilon}^{\otimes N}} \Big|^2\,dx_1\cdots dx_N \\
& +\frac{C}{\varepsilon^2} N^{-2 \alpha } + C N^{-\gamma }+\frac{C}{\varepsilon^4}N^{-1}+  C\|D^2\Phi^\varepsilon\|_{L^\infty(\mathbb{R}^2)}^2 N^{-2\alpha } + \frac{C}{\varepsilon^4} N^{-\gamma }.
\end{align*}
Taking into account the assumption $\|\nabla\log u^\varepsilon\|_{L^\infty(0,T;L^\infty(\mathbb{R}^2))} \leq C$, we proceed with $I_2$ as in \cite{BOL2026113712} immediately before (99).
\begin{align*}
I_2\leq& \frac{1}{2N}\sum_{i=1}^N\int_{\mathbb{R}^{2N}}u_N^\varepsilon\Big|\nabla_{x_i}\log \frac{u_N^\varepsilon}{{u^\varepsilon}^{\otimes N}} \Big|^2\,dx_1\cdots dx_N\nonumber\\
&+\frac{C}{N}\sum_{i=1}^N\mathbb{E}\Big(\Big|\frac{1}{N}\sum_{j=1}^N\big|\Phi^\varepsilon \left( (X_{N,i}^\varepsilon-X_{N,j}^\varepsilon)(t  ) \right)-
\Phi^\varepsilon \left( (\bar{X}_i^\varepsilon-\bar{X}_j^\varepsilon)(t  ) \right) \big|\Big|^2 \Big)\nonumber\\
&+\frac{C}{N}\sum_{i=1}^N\mathbb{E}\Big(\Big|\frac{1}{N}\sum_{j=1}^N\big(\Phi^\varepsilon \left( (\bar{X}_i^\varepsilon-\bar{X}_j^\varepsilon)(t  ) \right)
 -\Phi^\varepsilon\ast u^\varepsilon(t,\bar{X}_i^\varepsilon(t  ))\big)\Big|^2 \Big)\nonumber\\
 &+\frac{C}{N}\sum_{i=1}^N\mathbb{E}\big(|\Phi^\varepsilon\ast u^\varepsilon(t,\bar{X}_i^\varepsilon(t  ))-
 \Phi^\varepsilon\ast u^\varepsilon(t,X_{N,i}^\varepsilon(t  )) |^2 \big) \\
=: &  \frac{1}{2N}\sum_{i=1}^N\int_{\mathbb{R}^{2N}}u_N^\varepsilon\Big|\nabla_{x_i}\log \frac{u_N^\varepsilon}{{u^\varepsilon}^{\otimes N}} \Big|^2\,dx_1\cdots dx_N + I_{21} + I_{22} + I_{23}.
\end{align*}
We need to estimate $I_{21}$ and $I_{23}$. 

Since we already know from \eqref{est_D_phi_eps} that $ \Vert \nabla \Phi^\varepsilon \Vert_{L^\infty(\mathbb{R}^2)} \leq \frac{C}{\varepsilon^2}$ , we can conclude
\begin{align*}
& \mathbb{E}\bigg( \mathbbm{1}_{ {A_N^\varepsilon(t)}^c} \Big|\frac{1}{N}\sum_{j=1}^N\big|\Phi^\varepsilon \left( (X_{N,i}^\varepsilon-X_{N,j}^\varepsilon)(t  ) \right)-
\Phi^\varepsilon \left( (\bar{X}_i^\varepsilon-\bar{X}_j^\varepsilon)(t  ) \right) \big|\Big|^2 \bigg) \\
\leq & \frac{C}{\varepsilon^4} \mathbb{E}\Big( \mathbbm{1}_{ {A_N^\varepsilon(t)}^c} \max_{1 \leq i \leq N}  \big| \bar{X}_i^\varepsilon -X_{N,i}^\varepsilon \big|^2(t  ) \Big) \leq \frac{C}{\varepsilon^4} N^{-2\alpha }.
\end{align*}
From $\Vert \Phi^\varepsilon \Vert_{L^\infty( \R^2)} \leq \frac{C}{\varepsilon} $ see \eqref{est_phi_eps}) and theorem \ref{ks_conv_of_max_X-bar_X} we deduce
\begin{align*}
\mathbb{E}\bigg( \mathbbm{1}_{A_N^\varepsilon(t)} \Big|\frac{1}{N}\sum_{j=1}^N\big|\Phi^\varepsilon \left( (X_{N,i}^\varepsilon-X_{N,j}^\varepsilon)(t  ) \right)-
\Phi^\varepsilon \left( (\bar{X}_i^\varepsilon-\bar{X}_j^\varepsilon)(t  ) \right) \big|\Big|^2 \bigg)\leq \frac{C}{\varepsilon^2} \pr ( A_N^\varepsilon(t) ) \leq \frac{C}{\varepsilon^2} N^{- \gamma } .
\end{align*}
So we find
\begin{align*}
 I_{21} \leq & \frac{C}{\varepsilon^4} N^{-2\alpha } + \frac{C}{\varepsilon^2} N^{- \gamma }.
\end{align*}
We use the fact that $\|\nabla \Phi^\varepsilon\ast u^\varepsilon\|_{L^\infty((0,T)\times\mathbb{R}^2)} \leq C$ from \eqref{nablaPhistaru} and obtain
\begin{align*}
\mathbb{E}\Big( \mathbbm{1}_{ {A_N^\varepsilon(t)}^c} \big|\Phi^\varepsilon\ast u^\varepsilon(t,\bar{X}_i^\varepsilon(t  ))-
 \Phi^\varepsilon\ast u^\varepsilon(t,X_{N,i}^\varepsilon(t  )) \big|^2 \Big) \leq \mathbb{E}\Big( \mathbbm{1}_{ {A_N^\varepsilon(t)}^c}  C \big|\bar{X}_i^\varepsilon - X_{N,i}^\varepsilon \big|^2(t  ) \Big) \leq C N^{-2 \alpha}.
\end{align*}
In view of \eqref{prop_Phi}, we have $\Phi \in L^2(\R^2)$. The uniform bound of $u^\varepsilon$ from lemma \ref{bound_u^eps} leads to a uniform bound of for $\Phi^\varepsilon \ast u^\varepsilon$.
$$\|\Phi^\varepsilon \ast u^\varepsilon\|_{L^\infty((0,T)\times \mathbb{R}^2)} \leq \| \Phi^\varepsilon \|_{L^2(\mathbb{R}^2)} \| u^\varepsilon\|_{L^2((0,T)\times\mathbb{R}^2)} \leq C \| \Phi \|_{L^2(\mathbb{R}^2)} \| j^\varepsilon \|_{L^1(\mathbb{R}^2)} = C$$
By applying theorem \ref{ks_conv_of_max_X-bar_X} we get
\begin{align*}
\mathbb{E}\Big( \mathbbm{1}_{ A_N^\varepsilon(t)} |\Phi^\varepsilon\ast u^\varepsilon(t,\bar{X}_i^\varepsilon(t  ))-
 \Phi^\varepsilon\ast u^\varepsilon(t,X_{N,i}^\varepsilon(t  )) |^2 \Big) \leq  C\pr (A_N^\varepsilon(t)) \leq  CN^{-\gamma }.
\end{align*}
This yields
\begin{align*}
I_{23} \leq C N^{-2 \alpha } +  CN^{-\gamma }.
\end{align*}
Essentially, the same argument as in (57) and the previous steps in \cite[Prop 14]{BOL2026113712} shows $I_{22} \leq \frac{C}{\varepsilon^4} N^{-1} $. Therefore, we receive
\begin{align*}
I_2 \leq \frac{1}{2N}\sum_{i=1}^N\int_{\mathbb{R}^{2N}}u_N^\varepsilon\Big|\nabla_{x_i}\log \frac{u_N^\varepsilon}{{u^\varepsilon}^{\otimes N}} \Big|^2\,dx_1\cdots dx_N + \frac{C}{\varepsilon^4} N^{-2\alpha } + \frac{C}{\varepsilon^2} N^{- \gamma } + \frac{C}{\varepsilon^4} N^{-1} + CN^{-2 \alpha } +  CN^{-\gamma } .
\end{align*}
Hence, we obtain the estimate for the relative entropy in the following
\begin{align*}
& \frac{d}{dt}\mathcal{H}(u_N^\varepsilon|{u^\varepsilon}^{\otimes N}) \\
\leq & -\frac{1}{2N}\sum_{i=1}^N\int_{\mathbb{R}^{2N}}u_N^\varepsilon\Big|\nabla_{x_i}\log \frac{u_N^\varepsilon}{{u^\varepsilon}^{\otimes N}} \Big|^2\,dx_1\cdots dx_N\nonumber\\
&\qquad -\frac{1}{2N}\sum_{i=1}^N\int_{\mathbb{R}^{2N}}u_N^\varepsilon\exp\Big(-\frac{1}{N}\sum_{j=1}^N\Phi^\varepsilon(x_i-x_j) \Big)
\Big|\nabla_{x_i}\log\frac{u_N^\varepsilon}{{u^\varepsilon}^{\otimes N}} \Big|^2\,dx_1\cdots dx_N\nonumber\\
&\qquad +\frac{C}{\varepsilon^2} N^{-2 \alpha } + C N^{-\gamma }+\frac{C}{\varepsilon^4}N^{-1}+  C\|D^2\Phi^\varepsilon\|_{L^\infty(\mathbb{R}^2)}^2 N^{-2\alpha } + \frac{C}{\varepsilon^4} N^{-\gamma }\\
&\qquad  + \frac{C}{\varepsilon^4} N^{-2\alpha } + \frac{C}{\varepsilon^2} N^{- \gamma } + \frac{C}{\varepsilon^4} N^{-1} + C  N^{-2 \alpha } +  CN^{-\gamma }.
\end{align*}
\underline{From now on we assume} $\varepsilon \leq 1$. Moreover, we choose $\gamma = 2 \alpha$ and simplify the estimation by considering only the terms with the worst asymptotics.
$$ 
\frac{d}{dt} \mathcal{H}(u_N^\varepsilon|{u^\varepsilon}^{\otimes N}) \leq C\|D^2\Phi^\varepsilon\|_{L^\infty(\mathbb{R}^2)}^2 N^{-2\alpha } + \frac{C}{\varepsilon^4} N^{-\gamma } + \frac{C}{\varepsilon^4} N^{-2\alpha } \leq  C\|D^2\Phi^\varepsilon\|_{L^\infty(\mathbb{R}^2)}^2 N^{-2\alpha } + \frac{C}{\varepsilon^4} N^{-2\alpha }  
$$
In view of the scaling we have $\varepsilon^{2\frac{\alpha}{\beta}} \geq N^{-2 \alpha}$ and hence we find
$$
 \mathcal{H}(u_N^\varepsilon|{u^\varepsilon}^{\otimes N})(t) \leq  C\Big( \|D^2\Phi^\varepsilon\|_{L^\infty(\mathbb{R}^2)}^2 N^{-2\alpha } + \frac{1}{\varepsilon^4} N^{-2\alpha } \Big) T \leq  C\|D^2\Phi^\varepsilon\|_{L^\infty(\mathbb{R}^2)}^2 \varepsilon^{2\frac{\alpha}{\beta}} +  C\varepsilon^{2\frac{\alpha}{\beta}-4} . 
 $$
Now we need to discuss $\|D^2\Phi^\varepsilon\|_{L^\infty(\mathbb{R}^2)}$. Let $a>2$. Then \eqref{est_D^2_phi_eps} implies $\|D^2\Phi^\varepsilon\|_{L^\infty(\mathbb{R}^2)}^2 \varepsilon^{2\frac{\alpha}{\beta}} \leq C \varepsilon^{2(\frac{\alpha}{\beta}-a)}$. Observe that $2(\frac{\alpha}{\beta}-a) < 2\frac{\alpha}{\beta}-4$. We would like to denote the exponent with a single variable in a suitable range. Let $b \in (0, 2\frac{\alpha}{\beta}-4)$. It holds
\begin{align} \label{poc_strong_pf_a_b}
b=2\Big(\frac{\alpha}{\beta}-a \Big) \Leftrightarrow a=\frac{\alpha}{\beta} - \frac{b}{2}.
\end{align}
We can choose $a=\frac{\alpha}{\beta} - \frac{b}{2} >2$. Now \eqref{poc_strong_pf_a_b} implies $\|D^2\Phi^\varepsilon\|_{L^\infty(\mathbb{R}^2)}^2 \varepsilon^{2\frac{\alpha}{\beta}} \leq C \varepsilon^b$. Taking into account that $\varepsilon \leq 1$ we conclude
$$ 
\mathcal{H}(u_N^\varepsilon|{u^\varepsilon}^{\otimes N}) \leq C \varepsilon^b .
$$
Let $l \in \N$. \underline{From now on we assume} $N \geq l$. In the same way as in \cite[(101)]{BOL2026113712} it follows,
$$ \sup_{t\in [0,T]}\|u^\varepsilon_{N,l}(t)-{u^\varepsilon}^{\otimes l}(t)\|_{L^1(\mathbb{R}^{2l})}^2 \leq  C \varepsilon^{b}. $$ 
From the error estimates $(9)$ in \cite[Theorem 2]{BOL2026113712} we infer:
$$  \|{u^\varepsilon}^{\otimes l} - {u}^{\otimes l}\|_{L^\infty(0,T;L^1(\mathbb{R}^{2l}))}^2  \leq  C \varepsilon .$$
Finally we obtain
$$ \|u^\varepsilon_{N,l} - {u}^{\otimes l}\|_{L^\infty(0,T;L^1(\mathbb{R}^{2l}))}^2 \leq  C \max\big\{ \varepsilon^{b} , \varepsilon\big\}. $$ 
\end{proof}

\appendix

\counterwithin{theorem}{section}

\section{}

\begin{lemma}[Law of large numbers] \label{lln}
Let $(\Omega,\mathcal{F},\pr )$ be a probability space and $(Y_i : \Omega \to \R^d)_{i \in \N}$ a sequence of i.i.d random variables such that the distribution of $Y_i$ has probability density $v$. Moreover let $\Psi : \R^d \to \R$ continuous and bounded. We consider the following subsets:
\begin{align*}
& A^{N,i}_\theta (\Psi ,v):= \bigg\{ \Big|\displaystyle\frac{1}{N}\sum_{j=1}^N  \Psi (Y_i - Y_j) - \Psi \ast v (Y_i)  \Big| \geq \frac{1}{N^\theta} \bigg\} \mbox{ and } A^N_\theta (\Psi ,v) := \bigcup^N_{i=1} A^{N,i}_\theta (\Psi,v).
\end{align*}
For all $m \in \N$ and $\theta \in [0,\frac{1}{2})$ there exists some $C>0$ such that for all $N \in \N$
$$ \pr \big( A^N_\theta (\Psi ,v) \big) \leq N \max_{1 \leq i \leq N} \pr \big( A^{N,i}_\theta (\Psi,v) \big) \leq C N^{2m(\theta -\frac{1}{2})+1} \Vert \Psi \Vert_{L^\infty (\R^d)} .$$
\end{lemma}
The proof is similar to \cite[Lemma 7]{chen2024fluctuationsmeanfieldlimitattractive}.

\begin{lemma}[Uniform bound of $u^\varepsilon$] \label{bound_u^eps} Let $T>0$. The solution $u^\varepsilon$ of \eqref{kellersegelmedium} satisfies
\begin{align*}
\forall p \in [2, \infty ) : \exists C > 0 : \forall \varepsilon >0 : \Vert u^\varepsilon \Vert_{L^\infty (0,T;L^p(\R^d))} \leq C.
\end{align*}
\end{lemma}
This is a simple consequence of \cite[Proposition 11]{BOL2026113712}.

\bibliographystyle{abbrvdin}

\bibliography{lit}

\end{document}